\RequirePackage{fix-cm}
\documentclass[11pt, reqno]{amsart}
\usepackage{amsmath}
\usepackage{setspace}
\usepackage{fancyhdr}
\usepackage{mathrsfs}
\usepackage{xifthen}
\usepackage{verbatim}
\usepackage{hyperref}
\usepackage[all]{xcolor}

\usepackage{float}

\usepackage{comment}

\usepackage{geometry}
\usepackage{fancyhdr}
\usepackage{textcomp}
\usepackage{amsthm}
\usepackage{amstext}
\usepackage{amsfonts}
\usepackage{amssymb}
\usepackage{fixltx2e}
\usepackage{graphicx}

\usepackage{tikz}
\usepackage{tikz-cd}
\usepackage{caption}
\usetikzlibrary{calc,backgrounds}
\usetikzlibrary{matrix,arrows,decorations.pathmorphing}

\usepackage{bigstrut}

\makeatletter

\newcommand{\schff}{X}
\newcommand{\adicff}{\mathcal{X}}
\newcommand{\Spa}{\mathrm{Spa}}
\newcommand{\Spd}{\mathrm{Spd}}
\newcommand{\HN}{\mathrm{HN}}

\newcommand{\Surj}{\mathcal{S}\mathrm{urj}}
\newcommand{\Inj}{\mathcal{I}\mathrm{nj}}
\newcommand{\Hom}{\mathcal{H}\mathrm{om}}

\newcommand{\rk}{\mathrm{rk}}

\newcommand{\inj}{\hookrightarrow}
\newcommand{\surj}{\twoheadrightarrow}
\newcommand{\nonneg}{{\geq 0}}

\newcommand{\mumax}{\mu_\text{max}}
\newcommand{\mumin}{\mu_\text{min}}
\newcommand{\trivbundle}{\Ocal}

\newcommand{\vertstretch}{\tilde}
\newcommand{\rankred}{\breve}
\newcommand{\maxslopered}{\overline}

\newcommand{\nocommonslopered}{\acute}

\newcommand{\uniformizer}{\pi}
\newcommand{\pseudounif}{\varpi}
\newcommand{\finext}{E}
\newcommand{\algclosedperfdfield}{F}
\newcommand{\genperfdring}{R}

\newcommand{\integerring}[1]{{#1}^\circ}
\newcommand{\witt}{W}

\newcommand{\Perfd}{\mathrm{Perf}}

\numberwithin{equation}{section}

\usepackage{color}

\newcommand{\F}{\mathbb{F}}
\newcommand{\Q}{\mathbb{Q}}
\newcommand{\Z}{\mathbb{Z}}

\DeclareMathOperator{\rank}{rank}

\DeclareMathOperator{\Proj}{Proj\,}

\DeclareMathOperator{\dR}{dR}

\DeclareFontFamily{OT1}{rsfs}{}
\DeclareFontShape{OT1}{rsfs}{n}{it}{<-> rsfs10}{}
\DeclareMathAlphabet{\mathscr}{OT1}{rsfs}{n}{it}

\newcommand{\Dcal}{\mathcal{D}}
\newcommand{\Ecal}{\mathcal{E}}
\newcommand{\Fcal}{\mathcal{F}}

\newcommand{\Kcal}{\mathcal{K}}

\newcommand{\Mcal}{\mathcal{M}}

\newcommand{\Ocal}{\mathcal{O}}

\newcommand{\Qcal}{\mathcal{Q}}
\newcommand{\Rcal}{\mathcal{R}}
\newcommand{\Scal}{\mathcal{S}}
\newcommand{\Tcal}{\mathcal{T}}
\newcommand{\Ucal}{\mathcal{U}}
\newcommand{\Vcal}{\mathcal{V}}
\newcommand{\Wcal}{\mathcal{W}}

\newcommand{\Ycal}{\mathcal{Y}}

\newcommand{\customlabel}[2]{#2\def\@currentlabel{#2}\label{#1}}

\newtheorem{lemma}[subsubsection]{Lemma}
\newtheorem{prop}[subsubsection]{Proposition}

\theoremstyle{remark}
\newtheorem*{remark}{Remark}
\newtheorem{defn}[subsubsection]{Definition}
\newtheorem{example}[subsubsection]{Example}

\newtheorem*{thm*}{Theorem}
\makeatletter
\def\th@remark{%
  \thm@headfont{\bfseries}%
  \normalfont 
}
\def\imod#1{\allowbreak\mkern5mu({\operator@font mod}\,\,#1)}
\makeatother

\usepackage{enumitem}		
\theoremstyle{theorem}
\newtheorem{theorem}[subsubsection]{Theorem}

\numberwithin{equation}{section}

\makeatother

\begin{document}
	
	\tikzset{
		node style sp/.style={draw,circle,minimum size=\myunit},
		node style ge/.style={circle,minimum size=\myunit},
		arrow style mul/.style={draw,sloped,midway,fill=white},
		arrow style plus/.style={midway,sloped,fill=white},
	}
    
	\title{Classification of subbundles on the Fargues-Fontaine curve}
   
    \author[S. Hong]{Serin Hong}
    \address{Department of Mathematics, University of Michigan, 530 Church Street, Ann Arbor MI 48109}
    \email{serinh@umich.edu}
    
    \begin{abstract} We completely classify all subbundles of a given vector bundle on the Fargues-Fontaine curve. Our classification is given in terms of a simple and explicit condition on Harder-Narasimhan polygons. Our proof is inspired by the proof of the main theorem in \cite{Hong_quotvb}, but also involves a number of nontrivial adjustments. 

    \end{abstract}
	
	\maketitle

	\tableofcontents
	
	\rhead{}

	\chead{}
\section{Introduction}

\subsection{Motivation and background}\label{motivation}$ $

The \emph{Fargues-Fontaine curve} is a regular Noetherian scheme of Krull dimension $1$ which is constructed by Fargues-Fontaine \cite{FF_curve} as the fundamental curve of $p$-adic Hodge theory. 
Powered by the theory of perfectoid spaces and diamonds as developed by Scholze in \cite{Scholze_perfectoid} and \cite{Scholze_diamonds}, the theory of vector bundles on the Fargues-Fontaine curve has driven a number of spectacular discoveries in arithmetic geometry and $p$-adic Hodge theory. Notable examples include the geometrization of the local Langlands correspondence by Fargues \cite{Fargues_geomLL} and the construction of general local Shimura varieties by Scholze \cite{SW_berkeley}. 


One of the most fundamental results about vector bundles on the Fargues-Fontaine curve is that
they form a slope category which admits a complete classification by Harder-Narasimhan (HN) polygons, as we briefly recall below.

\begin{theorem}[Fargues-Fontaine \cite{FF_curve}, Kedlaya \cite{Kedlaya_slopefiltrations}] \label{classification of vector bundles on FF curve, intro} Fix a prime number $p$. Let $\finext$ be a finite extension of $\Q_p$, and let $\algclosedperfdfield$ be an algebraically closed perfectoid field of characteristic $p$. Denote by $\schff = \schff_{\finext, \algclosedperfdfield}$ the Fargues-Fontaine curve associated to the pair $(\finext, \algclosedperfdfield)$. 

\begin{enumerate}[label=(\arabic*)]
\item The category of vector bundles on $\schff$ admits a well-defined notion of slope. 
\smallskip

\item For every rational number $\lambda$, there is a unique stable bundle of slope $\lambda$ on $\schff$, denoted by $\trivbundle(\lambda)$. 
\smallskip

\item Every semistable bundle on $\schff$ of slope $\lambda$ is of the form $\trivbundle(\lambda)^{\oplus m}$. 
\smallskip

\item\label{HN decomp of vector bundles, intro} Every vector bundle $\Vcal$ on $\schff$ admits a (necessarily unique) Harder-Narasimhan decomposition
\[
\Vcal \simeq \bigoplus_i \trivbundle(\lambda_i)^{\oplus m_i}.
\]
In other words, the isomorphism class of $\Vcal$ is determined by the Harder-Narasimhan polygon $\HN(\Vcal)$ of $\Vcal$.
\end{enumerate}
\end{theorem}

Theorem \ref{classification of vector bundles on FF curve, intro} naturally leads to a question of classifying all quotient bundles and subbundles of a given vector bundle on the Fargues-Fontaine curve. For quotient bundles, the author in \cite{Hong_quotvb} has obtained a complete classification in terms of HN polygons. 
Our main purpose in this paper is to obtain a complete classification for subbundles.

We remark that the classification problem for subbundles is closely related to the study of \emph{modifications of vector bundles},
which play a pivotal role in studying the geometry of the $B_{\dR}$-affine Grassmannians, the flag varieties, and the Hecke stacks. 
By definition, a modification of vector bundles is an exact sequence of the form
\[ 0 \longrightarrow \Ecal \longrightarrow \Fcal \longrightarrow \Tcal \longrightarrow 0\]
for some vector bundles $\Ecal, \Fcal$ and torsion sheaf $\Tcal$. When $\Fcal$ is fixed, the vector bundles that 
can take place of $\Ecal$ are precisely subbundles of $\Fcal$ with maximal rank; therefore, once we have a complete classification for subbundles of $\Fcal$, we can describe all possible isomorphism classes of $\Ecal$. 



\subsection{Overview of the result}\label{overview}$ $

For a vector bundle $\Vcal$ on $\schff$ and a rational number $\mu$, we define a vector bundle $\Vcal^{\geq \mu}$ by declaring that its HN polygon $\HN(\Vcal^{\geq \mu})$ consists of all line segments in $\HN(\Vcal)$ with slope greater than or equal to $\mu$. In other words, for a vector bundle $\Vcal$ on $\schff$ with HN decomposition
\[\Vcal \simeq \bigoplus_i \trivbundle(\lambda_i)^{\oplus m_i},\]
we set 
\[ \Vcal^{\geq \mu}:= \bigoplus_{\lambda_i \geq \mu} \trivbundle(\lambda_i)^{\oplus m_i} \quad\quad\text{ for every } \mu \in \Q.\]
We can state our main result as follows:

\begin{theorem}\label{classification of subbundles, intro}
Let $\Fcal$ be a vector bundle on $\schff$. Then a vector bundle $\Ecal$ on $\schff$ is a subbundle of $\Fcal$ if and only if the following equivalent conditions are satisfied:
\begin{enumerate}[label=(\roman*)]
\item\label{rank inequalities for subbundles, intro} $\rank(\Ecal^{\geq \mu}) \leq \rank(\Fcal^{\geq \mu})$ for every $\mu \in \Q$. 

\smallskip

\item\label{slopewise dominance for subbundles, intro} For each $i = 1, \cdots, \rank(\Ecal)$, the slope of $\HN(\Ecal)$ on the interval $[i-1, i]$ is less than or equal to the slope of $\HN(\Fcal)$ on this interval. 
\end{enumerate}

\begin{figure}[H]
\begin{tikzpicture}[scale=1]

		\coordinate (left) at (0, 0);
		\coordinate (q0) at (1,2.5);
		\coordinate (q1) at (2.5, 4);
		\coordinate (q2) at (5, 5);
		\coordinate (q3) at (8, 4);
		

		\coordinate (p0) at (2, 1.5);
		\coordinate (p1) at (4.5, 2);
		\coordinate (p2) at (6.5, 0.7);
				
		\draw[step=1cm,thick] (left) -- (q0) --  (q1) -- (q2) -- (q3);
		\draw[step=1cm,thick] (left) -- (p0) --  (p1) -- (p2);
		
		\draw [fill] (q0) circle [radius=0.05];		
		\draw [fill] (q1) circle [radius=0.05];		
		\draw [fill] (q2) circle [radius=0.05];		
		\draw [fill] (q3) circle [radius=0.05];		
		\draw [fill] (left) circle [radius=0.05];
		
		\draw [fill] (p0) circle [radius=0.05];		
		\draw [fill] (p1) circle [radius=0.05];		
		\draw [fill] (p2) circle [radius=0.05];		
		
		\draw[step=1cm,dotted] (3, -0.4) -- (3, 5);
       		\draw[step=1cm,dotted] (3.5, -0.4) -- (3.5, 5);

		\node at (2.8,-0.8) {\scriptsize $i-1$};
		\node at (3.5,-0.8) {\scriptsize $i$};
		
		\path (q3) ++(0.8, 0.05) node {$\HN(\Fcal)$};
		\path (p2) ++(0.8, 0.05) node {$\HN(\Ecal)$};
		\path (left) ++(-0.3, -0.05) node {$O$};

\end{tikzpicture}
\caption{Illustration of the condition \ref{slopewise dominance for subbundles} in Theorem \ref{classification of subbundles}.}
\end{figure}
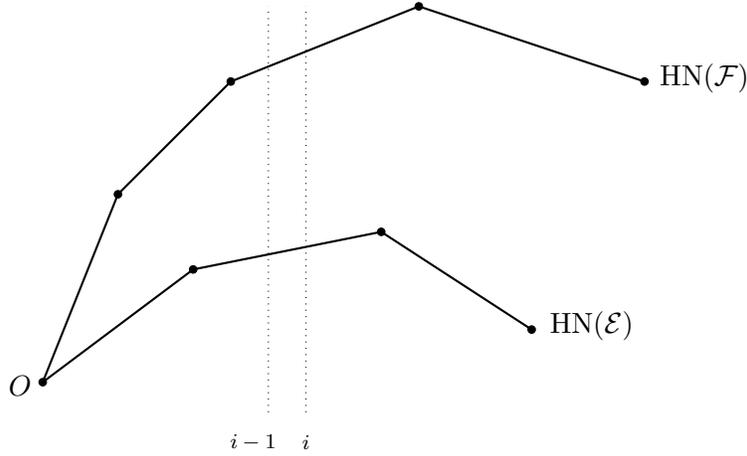
\end{theorem}

Let us briefly sketch our proof of Theorem \ref{classification of subbundles, intro}, which is largely inspired by the main argument in \cite{Hong_quotvb} for classification of quotient bundles. The necessity part of Theorem \ref{classification of subbundles, intro} is a direct consequence of the slope formalism, while equivalence of the conditions \ref{rank inequalities for subbundles, intro} and \ref{slopewise dominance for subbundles, intro} follows immediately from convexity of HN polygons. Thus the main part of the proof will concern the sufficiency part of Theorem \ref{classification of subbundles, intro}. To this end, we will consider auxiliary moduli spaces $\Hom(\Ecal, \Fcal)_\Qcal$ which (roughly) parametrize bundle maps $\Ecal \to \Fcal$ with image isomorphic to a specified vector bundle $\Qcal$. These spaces are \emph{diamonds} in the sense of Scholze \cite{Scholze_diamonds}, as shown in \cite{Arizona_extvb}. For the assertion that $\Ecal$ is a subbundle of $\Fcal$, it suffices to prove nonemptiness of $\Hom(\Ecal, \Fcal)_\Ecal$. Using the dimension theory for diamonds, we will reduce the desired nonemptiness of $\Hom(\Ecal, \Fcal)_\Ecal$ to a quantitative statement as stated in Proposition \ref{key inequality}. Then we will prove this quantitative statement by 
a certain degenerating process on the dual bundle of $\Ecal$. 

In many parts, our argument will adapt various notions and constructions from \cite{Hong_quotvb}. Most notably, the notion of \emph{slopewise dominance} as defined in \cite{Hong_quotvb} will play a crucial role in both the formulation and the proof of the key quantitative statement, namely Proposition \ref{key inequality}. In addition, our degenerating process on the dual bundle of $\Ecal$ will be almost identical to the degenerating process on $\Fcal$ in the main argument of \cite{Hong_quotvb}.

However, the details of our argument will require several nontrivial adjustments from the argument in \cite{Hong_quotvb}. These adjustments essentially come from the fact that Theorem \ref{classification of subbundles, intro} cannot be deduced by simply dualizing the classification theorem for quotient bundles as obtained in \cite{Hong_quotvb}. 
In our proof, we will try to indicate what adjustments we should make and why such adjustments are necessary.





\subsection*{Acknowledgments} The author would like to thank Adrian Vasiu for a valuable discussion about the problem. The author would also like to thank Tasho Kaletha for his helpful suggestion about the manuscript. 

\newpage

\section{Vector bundles on the Fargues-Fontaine curve}\label{background}

\subsection{The Fargues-Fontaine curve}$ $

Throughout this paper, we fix the following data:
\begin{itemize}
\item $p$ is a prime number;

\item $\finext$ is a finite extension of $\Q_p$ with residue field $\F_q$;

\item $\algclosedperfdfield$ is an algebraically closed perfectoid field of characteristic $p$. 
\end{itemize}
In addition, we denote by $\integerring{\finext}$ and $\integerring{\algclosedperfdfield}$ the rings of integers of $\finext$ and $\algclosedperfdfield$, respectively. We also choose a uniformizer $\uniformizer$ of $\finext$ and a pseudouniformizer $\pseudounif$ of $\algclosedperfdfield$. 

Let $\witt_{\finext^\circ}(\integerring{\algclosedperfdfield}):=\witt(\integerring{\algclosedperfdfield}) \otimes_{\witt(\F_q)} \integerring{\finext}$ be the ring  of ramified Witt vectors of $\integerring{\algclosedperfdfield}$ with coefficients in $\integerring{\finext}$, and let $[\pseudounif]$ be the Teichmuller lift of $\pseudounif$. One can show that 
\[
\Ycal:=\Spa(\witt_{\integerring{\finext}}(\integerring{\algclosedperfdfield}))\setminus\{|p[\pseudounif]|=0\}
\]
is an adic space over $\Spa(\finext)$. Moreover, the natural $q$-Frobenius map on $\witt_{\integerring{\finext}}(\integerring{\algclosedperfdfield})$ induces a properly discontinuous automorphism $\phi$ of $\Ycal$. 
\begin{defn} Given the pair $(\finext, \algclosedperfdfield)$, we define the associated \emph{adic Fargues-Fontaine curve} by
\[
\adicff:=\Ycal/\phi^\Z,
\]
and the associated \emph{schematic Fargues-Fontaine curve} by 
\[\schff:= \Proj \left( \bigoplus_{n \geq 0 } H^0(\Ycal, \trivbundle_\Ycal)^{\phi = \pseudounif^n} \right).\]
\end{defn}

In this paper, we will speak interchangeably about vector bundles on $\adicff$ and $\schff$. There will be no harm from doing this because of the following GAGA type result:
\begin{prop}[``GAGA for the Fargues-Fontaine curve'', {\cite[Theorem 6.3.12]{KL15}}]\label{GAGA for FF curve}
There is a natural map of locally ringed spaces 
\[
\adicff \rightarrow \schff
\]
which induces by pullback an equivalence of categories of vector bundles. 
\end{prop}

The Fargues-Fontaine curve is a ``curve" in the following sense:

\begin{prop}[\cite{FF_curve}]\label{FF curve is a curve} The scheme $\schff$ is a regular, Noetherian scheme over $\finext$ of Krull dimension $1$.  
\end{prop}

\begin{remark}
However, the scheme $\schff$ is not of finite type over $\finext$. In fact, the residue field at a closed point is a complete algebraically closed extension of $\finext$. 
\end{remark}


We can extend the construction of the Fargues-Fontaine curve to relative settings. Let $S = \Spa (\genperfdring, \genperfdring^+)$ be an affinoid perfectoid space over $\Spa (\algclosedperfdfield)$, and let $\pseudounif_\genperfdring$ be a pseudouniformizer  of $\genperfdring$. We take the ring of ramified Witt vectors $\witt_{\finext^\circ}(\genperfdring^+):=\witt(\genperfdring^+) \otimes_{\witt(\F_q)} \integerring{\finext}$ and write $[\pseudounif_\genperfdring]$ for the Teichmuller lift of $\pseudounif_\genperfdring$. One can show that 
\[
\Ycal_S:=\Spa(\witt_{\integerring{\finext}}(\genperfdring^+), \witt_{\integerring{\finext}}(\genperfdring^+))\setminus\{|p[\pseudounif_\genperfdring]|=0\}
\]
is an adic space over $\Spa(\finext)$, equipped with a properly discontinuous automorphism $\phi$ induced by the natural $q$-Frobenius on $\witt_{\integerring{\finext}}(\genperfdring^+)$. 

\begin{defn}
Given an affinoid perfectoid space $S = \Spa (\genperfdring, \genperfdring^+)$ over $\Spa (\algclosedperfdfield)$, we define the \emph{adic Fargues-Fontaine curve} associated to the pair $(\finext, S)$ by
\[
\adicff_S:=\Ycal_S/\phi^\Z,
\]
and the \emph{schematic Fargues-Fontaine curve} associated to the pair $(\finext, S)$ by 
\[\schff_S:= \Proj \left( \bigoplus_{n \geq 0 } H^0(\Ycal_S, \trivbundle_\Ycal)^{\phi = \pseudounif^n} \right).\]
More generally, for an arbitrary perfectoid space $S$ over $\Spa(\algclosedperfdfield)$, we choose an affinoid cover $S = \bigcup S_i = \bigcup \Spa(\genperfdring_i, \genperfdring_i^+)$ and define the adic Fargues-Fontaine curve $\adicff_S$ and the schematic Fargues-Fontaine curve $\schff_S$ respectively by gluing the $\adicff_{S_i}$ and the $\schff_{S_i}$. 
\end{defn}

There is a GAGA type result which extends Proposition \ref{GAGA for FF curve} to relative settings, thereby allowing us to speak interchangeably about vector bundles on $\adicff_S$ and $\schff_S$ for any perfectoid space $S$ over $\Spa(\algclosedperfdfield)$.

\subsection{Slope theory for vector bundles}$ $

In this subsection we briefly review the slope theory for vector bundles on the Fargues-Fontaine curve. 

\begin{defn}
Given a vector bundle $\Vcal$ on $\schff$, we write $\rk(\Vcal)$ for the rank of $\Vcal$ and $\Vcal^\vee$ for the dual bundle of $\Vcal$. 
\end{defn}

The Fargues-Fontaine curve $\schff$ is not a complete curve; in fact, as remarked after Proposition \ref{FF curve is a curve}, it is not even of finite type. Nonetheless, the Fargues-Fontaine curve behaves as a complete curve in the following sense:
\begin{prop}[\cite{FF_curve}]\label{completeness of FF curve}
For an arbitrary nonzero rational function $f$ on $\schff$, its divisor $\mathrm{div}(f)$ has degree zero. 
\end{prop}

We thus have a well-defined notion of degree and slope for vector bundles on $\schff$. 
\begin{defn}
Let $\Vcal$ be a vector bundle on $\schff$. 
\begin{enumerate}[label=(\arabic*)]
\item If $\Vcal$ is a line bundle (i.e., $\rk(\Vcal) = 1$), we define the \emph{degree} of $\Vcal$ by
\[ \deg(\Vcal): = \deg(\mathrm{div}(s))\]
where $s$ is an arbitrary nonzero meromorphic section of $\Vcal$. In general, we define 
\[\deg (\Vcal) := \deg (\wedge^{\rk(\Vcal)} \Vcal).\]
\item We define the \emph{slope} of $\Vcal$ by
\[\mu(\Vcal) := \dfrac{\deg(\Vcal)}{\rk(\Vcal)}.\]
\end{enumerate}
\end{defn}

We explicitly construct some vector bundles on $\schff$ which will serve as building blocks for general vector bundles on $\schff$. Let $\lambda = r/s$ be a rational number written in lowest terms with $r>0$. We choose a trivializing basis $v_1, v_2, \cdots, v_s$ of $\trivbundle_\Ycal^{\oplus s}$, and define an isomorphism $\phi^* \trivbundle_\Ycal^{\oplus s} \stackrel{\sim}{\longrightarrow} \trivbundle_\Ycal^{\oplus s}$ by
\[ v_1 \mapsto v_2, \quad v_2 \mapsto v_3, \quad\cdots, \quad v_{s-1} \mapsto v_s, \quad v_s \mapsto \uniformizer^{-r} v_1,\]
where we abuse notation to view $v_1, v_2, \cdots, v_s$ as a trivializing basis for $\phi^* \trivbundle_\Ycal^{\oplus s}$ as well. We denote by $\widetilde{\trivbundle(\lambda)}$ the vector bundle $\trivbundle_\Ycal^{\oplus s}$ equipped with the isomorphism $\phi^* \trivbundle_\Ycal^{\oplus s} \stackrel{\sim}{\longrightarrow} \trivbundle_\Ycal^{\oplus s}$ as defined above. 


\begin{defn}\label{o-r-over-s}
Given a rational number $\lambda$, we write $\trivbundle(\lambda)$ for the vector bundle on $\adicff$ obtained by descending the vector bundle $\widetilde{\trivbundle(\lambda)}$, and also for the corresponding vector bundle on $\schff$ under the GAGA functor described in Proposition \ref{GAGA for FF curve}. 
\end{defn}

\begin{lemma}\label{rank, degree and dual of stable bundles}
Let $\lambda = r/s$ be a rational number written in lowest terms with $r>0$. 
\begin{enumerate}[label=(\arabic*)]
\item\label{rank and degree of o-lambda} $\rk(\trivbundle(\lambda)) = s$ and $\deg(\trivbundle(\lambda)) = r$. 
\smallskip

\item $\trivbundle(\lambda)^\vee \simeq \trivbundle(-\lambda)$. 
\end{enumerate}
\end{lemma}

\begin{proof}
All statements are straightforward to check using Definition \ref{o-r-over-s}.
\end{proof}

Let us now recall the notions of semistability and stability. 

\begin{defn} A vector bundle $\Vcal$ on $\schff$ is \emph{semistable} if every nonzero proper subbundle $\Wcal$ of $\Vcal$ satisfies
\begin{equation}\label{semistability slope ineq}
\mu(\Wcal) \leq \mu(\Vcal).
\end{equation}
A semistable vector bundle $\Vcal$ on $\schff$ is \emph{stable} if the equality in \eqref{semistability slope ineq} never holds. 
\end{defn}

It turns out that the category of vector bundles on $\schff$ admits an explicit characterization of stability and semistability, as well as a complete classification of isomorphism classes, in terms of the vector bundles that we constructed in Definition \ref{o-r-over-s}.

\begin{theorem}[\cite{FF_curve}]\label{existence of HN decomp}
Let $\Vcal$ be a vector bundle on $\schff$. 
\begin{enumerate}[label = (\arabic*)]
\item $\Vcal$ is stable of slope $\lambda$ if and only if $\Vcal \simeq \trivbundle(\lambda)$. 

\item $\Vcal$ is semistable of slope $\lambda$ if and only if $\Vcal \simeq \trivbundle(\lambda)^{\oplus n}$ for some $n$. 

\item In general, $\Vcal$ admits a unique direct sum decomposition of the form 
\begin{equation}\label{HN decomposition}
\Vcal \simeq \bigoplus_{i=1}^l \trivbundle(\lambda_i)^{\oplus m_i}
\end{equation}
where $\lambda_1 > \lambda_2 > \cdots > \lambda_l$. 
\end{enumerate}

\end{theorem}

\begin{defn}\label{def of HN decomposition and HN polygon}
Let $\Vcal$ be a vector bundle on $\schff$. 
\begin{enumerate}[label = (\arabic*)]
\item We refer to the decomposition \eqref{HN decomposition} 
in Theorem \ref{existence of HN decomp} 
as the \emph{Harder-Narasimhan (HN) decomposition} of $\Vcal$. 
\smallskip

\item We refer to the slopes $\lambda_i$ of direct summands in the HN decomposition as the \emph{Harder-Narasimhan (HN) slopes} of $\Vcal$, or often simply as the \emph{slopes} of $\Vcal$.
\smallskip

\item We write $\mumax(\Vcal)$ (resp. $\mumin(\Vcal)$) for the maximum (resp. minimum) HN slope of $\Vcal$. In other words, we set
\[ \mumax(\Vcal) := \lambda_1 \quad\quad \text{ and } \quad\quad \mumin(\Vcal) := \lambda_l.\]

\item For every $\mu \in \Q$ we set
\[\Vcal^{\geq \mu} := \bigoplus_{\lambda_i \geq \mu}\trivbundle(\lambda_i)^{\oplus m_i} \quad\quad \text{ and } \quad\quad\Vcal^{\leq \mu} := \bigoplus_{\lambda_i \leq \mu}\trivbundle(\lambda_i)^{\oplus m_i},\]
and similarly define $\Vcal^{>\mu}$ and $\Vcal^{<\mu}$. 
\smallskip

\item We define the \emph{Harder-Narasimhan (HN) polygon} of $\Vcal$ as the upper convex hull of the points $(0, 0)$ and $\big(\rk(\Vcal^{\geq \lambda_i}), \deg(\Vcal^{\geq \lambda_i})\big)$
\end{enumerate}
\end{defn}

We collect some basic facts about the slope theory for vector bundles on $\schff$. 

\begin{prop}\label{classification by HN polygon}
The isomorphism class of every vector bundle $\Vcal$ on $\schff$ is determined by the HN polygon $\HN(\Vcal)$. 
\end{prop}
\begin{proof}
Let us write the HN decomposition of $\Vcal$ as
\[\Vcal \simeq \bigoplus_{i=1}^l \trivbundle(\lambda_i)^{\oplus m_i}\]
where $\lambda_1 > \lambda_2 > \cdots > \lambda_l$. It suffices to show that we can find the numbers $\lambda_i$ and $m_i$ from $\HN(\Vcal)$. 

Let $x_i$ and $y_i$ respectievly denote the horizontal and vertical length of the $i$-th segment in $\HN(\Vcal)$. From definition we find
\begin{align*}
x_i &= \rk(\Vcal^{\geq \lambda_i}) - \rk(\Vcal^{\geq \lambda_{i-1}}) = \rk(\trivbundle(\lambda_i)^{\oplus m_i}) = m_i \cdot \rk(\trivbundle(\lambda_i)), \\
y_i &= \deg(\Vcal^{\geq \lambda_i}) - \deg(\Vcal^{\geq \lambda_{i-1}}) = \deg(\trivbundle(\lambda_i)^{\oplus m_i}) = m_i \cdot \deg(\trivbundle(\lambda_i)).
\end{align*}
We thus find $\lambda_i$ by
\[\lambda_i = \mu(\trivbundle(\lambda_i)) = \dfrac{\deg(\trivbundle(\lambda_i))}{\rk(\trivbundle(\lambda_i))} = \dfrac{y_i}{x_i}.\]
We then obtain $\rk(\trivbundle(\lambda_i))$ and $\deg(\trivbundle(\lambda_i))$ by Lemma \ref{rank, degree and dual of stable bundles}, and in turn find $m_i$ by
\[ m_i = \dfrac{x_i}{\rk(\trivbundle(\lambda_i))} = \dfrac{y_i}{\deg(\trivbundle(\lambda_i))}.\]
\end{proof}

\begin{lemma}\label{rank and degree of dual bundle}
Let $\Vcal$ be a vector bundle on $\schff$. We have identities
\[ \rk(\Vcal) = \rk(\Vcal^\vee) \quad\text{ and }\quad \deg(\Vcal) = -\deg(\Vcal^\vee).\]
Moreover, for every $\mu \in \Q$ we have identities
\[\rk(\Vcal^{\geq \mu}) = \rk((\Vcal^\vee)^{\leq-\mu}) \quad\text{ and }\quad \deg(\Vcal^{\geq \mu}) = -\deg((\Vcal^\vee)^{\leq-\mu}).\]
\end{lemma}
\begin{proof}
We verify the first statement for stable $\Vcal$ by Lemma \ref{rank, degree and dual of stable bundles}, then extend it to general $\Vcal$ using the HN decomposition. We then deduce the second statement from the first statement by observing $(\Vcal^{\geq \mu})^\vee \simeq (\Vcal^\vee)^{\leq -\mu}$ using Lemma \ref{rank, degree and dual of stable bundles}. 
\end{proof}

\begin{lemma}\label{zero hom for dominating slopes}
Given two vector bundles $\Vcal$ and $\Wcal$ on $\schff$ with $\mumin(\Vcal)>\mumax(\Wcal)$, we have
\[\mathrm{Hom}(\Vcal, \Wcal) = 0.\]
\end{lemma}
\begin{proof}
Using the HN decomposition, we immediately reduce to the case when both $\Vcal$ and $\Wcal$ are stable. Note that the condition $\mumin(\Vcal)>\mumax(\Wcal)$ now becomes $\mu(\Vcal) > \mu(\Wcal)$. 

Suppose for contradiction that there exists a nonzero bundle map $f: \Vcal \longrightarrow \Wcal$. Let $\Qcal$ denote the image of this map, which is nonzero by our assumption. Since $\Qcal$ is a subbundle of $\Wcal$, stability of $\Wcal$ yields
\begin{equation}\label{stable bundle maps slope ineq for image and target}
\mu(\Qcal) \leq \mu(\Wcal).
\end{equation} 
On the other hand, the surjective bundle map $\Vcal \surj \Qcal$ gives an injective dual map $\Qcal^\vee \inj \Vcal^\vee$. Since stability of $\Vcal$ implies stability of $\Vcal^\vee$ by Lemma \ref{rank, degree and dual of stable bundles}, we obtain an inequality
\[\mu(\Qcal^\vee) \leq \mu(\Vcal^\vee).\]
By Lemma \ref{rank and degree of dual bundle}, this inequality is equivalent to
\begin{equation}\label{stable bundle maps slope ineq for source and image}
\mu(\Vcal) \leq \mu(\Qcal).
\end{equation} 
Now we combine \eqref{stable bundle maps slope ineq for image and target} and \eqref{stable bundle maps slope ineq for source and image} to find $\mu(\Vcal) \leq \mu(\Wcal)$, thereby completing the proof by contradiction. 
\end{proof}

\subsection{Moduli of bundle maps}\label{bundlemaps}$ $

In this subsection we define certain moduli spaces of bundle maps over $\schff$ and discuss some of their key properties. The reader can find a detailed discussion about these spaces in \cite[\S3.3]{Arizona_extvb}. 




Let us first define these moduli spaces as functors on the category of perfectoid spaces over $\Spa(\algclosedperfdfield)$, which we denote by $\Perfd_{/\Spa (\algclosedperfdfield)}$. Note that, by construction, the relative Fargues-Fontaine curve $\schff_S$ for any $S \in \Spa(\algclosedperfdfield)$ comes with a natural map $\schff_S \to \schff$. 

\begin{defn} Let $\Ecal$ and $\Fcal$ be vector bundles on $\schff$. For any $S \in \Spa(\algclosedperfdfield)$, we denote by $\Ecal_S$ and $\Fcal_S$ the vector bundles on $\schff_S$ obtained as the pullback of $\Ecal$ and $\Fcal$ along the map $\schff_S \to \schff$. 
\begin{enumerate}
\item $\Hom(\mathcal{E},\Fcal)$ is the functor
associating $S \in \Perfd_{/\Spa(\algclosedperfdfield)}$ to the set of $\trivbundle_{\schff_S}$-module
maps $m:\Ecal_S\to\Fcal_S$. 
\smallskip

\item $\Surj(\Ecal, \Fcal)$ is the functor associating $S \in \Perfd_{/\Spa(\algclosedperfdfield)}$ to the set of surjective $\trivbundle_{\schff_S}$-module maps $m: \Ecal_S \surj \Fcal_S$. 
\smallskip

\item $\Inj(\Ecal, \Fcal)$ is the functor associating $S \in \Perfd_{/\Spa(\algclosedperfdfield)}$ to the set of 
$\trivbundle_{\schff_S}$-module
maps $m:\Ecal_{S}\to\Fcal_{S}$ whose pullback along the map $\schff_{\overline{x}} \to\schff_S$ for any geometric point $\overline{x}
\to S$
gives an injective $\trivbundle_{\adicff_{\overline{x}}}$-module map. %


\end{enumerate}
\end{defn}



It turns out that we can make sense of these functors as moduli spaces in the category of diamonds as defined by Scholze \cite{Scholze_diamonds}. 

\begin{prop}[{\cite[Proposition 3.3.2, Proposition 3.3.5 and Proposition 3.3.6]{Arizona_extvb}}]\label{moduli of bundle maps fund facts}
Let $\Ecal$ and $\Fcal$ be vector bundles on $\schff$. The functors $\Hom(\Ecal, \Fcal)$, $\Surj(\Ecal, \Fcal)$ and $\Inj(\Ecal, \Fcal)$ are all locally spatial and partially proper diamonds 
in the sense of Scholze \cite{Scholze_diamonds}. Moreover, we have the following facts:
\begin{enumerate}[label=(\arabic*)]
\item The diamond $\Hom(\Ecal, \Fcal)$ is equidimensional of dimension $\deg(\Ecal^\vee \otimes \Fcal)^\nonneg$. 
\smallskip

\item Every nonempty open subfunctor of $\Hom(\Ecal, \Fcal)$ has an $\algclosedperfdfield$-point. 
\smallskip

\item The diamonds $\Surj(\Ecal, \Fcal)$ and $\Inj(\Ecal, \Fcal)$ are both open subfunctors of $\Hom(\Ecal, \Fcal)$. 

\end{enumerate}
\end{prop}

\begin{remark}
The functor $\Hom(\Ecal, \Fcal)$ is also a Banach-Colmez space as defined by Colmez \cite{Colmez_BCspace}. Moreover, its dimension as a diamond coincides with its ``principal" dimension as a Banach-Colmez space. 
\end{remark}

\begin{defn}
We write $|\Hom(\Ecal, \Fcal)|, |\Surj(\Ecal, \Fcal)|$ and $|\Inj(\Ecal, \Fcal)|$ respectively for the underlying topological space of the diamonds $\Hom(\Ecal, \Fcal), \Surj(\Ecal, \Fcal)$ and $\Inj(\Ecal, \Fcal)$. 
\end{defn}



For the proof of our main theorem, we will consider a stratification of the $\Hom$ space according to the isomorphism type of image. 
\begin{defn}\label{Hom with specified image} Given vector bundles $\Ecal, \Fcal$ and $\Qcal$ on $\schff$, we define $\Hom(\Ecal, \Fcal)_\Qcal$ as the image of 
the map of diamonds
\[\Surj(\Ecal,\Qcal) \times_{\Spd\,\algclosedperfdfield} \Inj(\Qcal,\Fcal) \to \Hom(\Ecal,\Fcal)\] 
induced by composition of bundle maps, and denote by $|\Hom(\Ecal, \Fcal)_{\Qcal}|$ the underlying topological space of $\Hom(\Ecal, \Fcal)_\Qcal$.
\end{defn}


\begin{prop}[{\cite[Proposition 3.3.9 and Lemma 3.3.10]{Arizona_extvb}}]\label{properties of Hom with specified image}
Given vector bundles $\Ecal, \Fcal$ and $\Qcal$ on $\schff$, the topological space $|\Hom(\Ecal, \Fcal)_{\Qcal}|$ satisfies the following properties:
\begin{enumerate}[label=(\arabic*)]
\item\label{Hom with specified image behave nicely} $|\Hom(\Ecal, \Fcal)_{\Qcal}|$ is stable under generalization and specialization inside $|\Hom(\Ecal, \Fcal)|$. 
\smallskip

\item\label{dim formula for Hom with specified image} If $|\Hom(\Ecal, \Fcal)_{\Qcal}|$ is nonempty, its dimension is given by
\[\dim |\Hom(\Ecal, \Fcal)_{\Qcal}| = \deg(\Ecal^\vee \otimes \Qcal)^\nonneg + \deg(\Qcal^\vee \otimes \Fcal)^\nonneg - \deg(\Qcal^\vee \otimes \Qcal)^\nonneg.\]
\end{enumerate}
\end{prop}

\section{Classification of subbundles}
  
\subsection{The main theorem and primary reductions}\label{statement and primary reduction}$ $

The rest of this paper will be devoted to establishing our main result as stated below.  


\begin{theorem}\label{classification of subbundles}
Let $\Fcal$ be a vector bundle on $\schff$. Then a vector bundle $\Ecal$ on $\schff$ is a subbundle of $\Fcal$ if and only if the following equivalent conditions are satisfied:
\begin{enumerate}[label=(\roman*)]
\item\label{rank inequalities for subbundles} $\rk(\Ecal^{\geq \mu}) \leq \rk(\Fcal^{\geq \mu})$ for every $\mu \in \Q$. 

\smallskip

\item\label{slopewise dominance for subbundles} For each $i = 1, \cdots, \rank(\Ecal)$, the slope of $\HN(\Ecal)$ on the interval $[i-1, i]$ is less than or equal to the slope of $\HN(\Fcal)$ on this interval. 
\end{enumerate}

\begin{figure}[H]
\begin{tikzpicture}[scale=1]

		\coordinate (left) at (0, 0);
		\coordinate (q0) at (1,2.5);
		\coordinate (q1) at (2.5, 4);
		\coordinate (q2) at (5, 5);
		\coordinate (q3) at (8, 4);
		

		\coordinate (p0) at (2, 1.5);
		\coordinate (p1) at (4.5, 2);
		\coordinate (p2) at (6.5, 0.7);
				
		\draw[step=1cm,thick] (left) -- (q0) --  (q1) -- (q2) -- (q3);
		\draw[step=1cm,thick] (left) -- (p0) --  (p1) -- (p2);
		
		\draw [fill] (q0) circle [radius=0.05];		
		\draw [fill] (q1) circle [radius=0.05];		
		\draw [fill] (q2) circle [radius=0.05];		
		\draw [fill] (q3) circle [radius=0.05];		
		\draw [fill] (left) circle [radius=0.05];
		
		\draw [fill] (p0) circle [radius=0.05];		
		\draw [fill] (p1) circle [radius=0.05];		
		\draw [fill] (p2) circle [radius=0.05];		
		
		\draw[step=1cm,dotted] (3, -0.4) -- (3, 5);
       		\draw[step=1cm,dotted] (3.5, -0.4) -- (3.5, 5);

		\node at (2.8,-0.8) {\scriptsize $i-1$};
		\node at (3.5,-0.8) {\scriptsize $i$};
		
		\path (q3) ++(0.8, 0.05) node {$\HN(\Fcal)$};
		\path (p2) ++(0.8, 0.05) node {$\HN(\Ecal)$};
		\path (left) ++(-0.3, -0.05) node {$O$};

\end{tikzpicture}
\caption{Illustration of the condition \ref{slopewise dominance for subbundles} in Theorem \ref{classification of subbundles}.}
\end{figure}
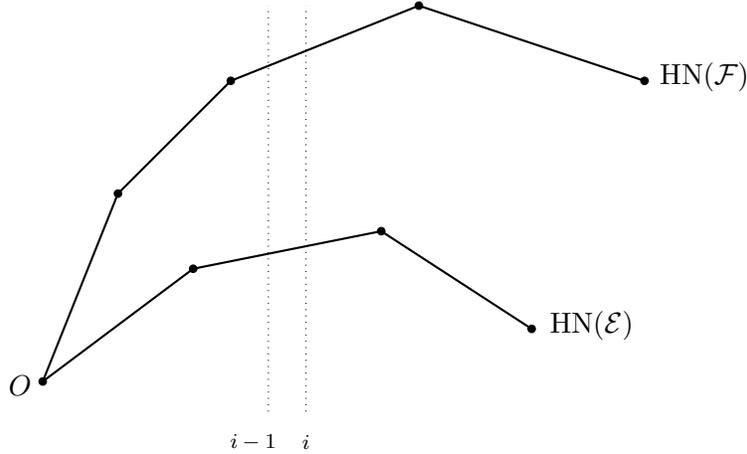
\end{theorem}

We begin our proof of Theorem \ref{classification of subbundles} by proving necessity of the condition \ref{rank inequalities for subbundles}. 


\begin{prop}\label{subbundles necessary condition}
Given a vector bundle $\Fcal$ on $\schff$, every subbundle $\Ecal$ of $\Fcal$ satisfies the condition \ref{rank inequalities for subbundles} in Theorem \ref{classification of subbundles}. 
\end{prop}

\begin{proof}
Let $\mu$ be an arbitrary rational number, and choose an injective bundle map $\Ecal \inj \Fcal$. Lemma \ref{zero hom for dominating slopes} implies that this map should embed $\Ecal^{\geq \mu}$ into $\Fcal^{\geq \mu}$, thereby yielding the desired inequality $\rk(\Ecal^{\geq \mu}) \leq \rk(\Fcal^{\geq \mu})$. 
\end{proof}

For sufficiency of the condition \ref{rank inequalities for subbundles}, we note the following easy but important reduction.
\begin{prop}\label{reduction on common slopes}
We may prove sufficiency of the condition \ref{rank inequalities for subbundles} in Theorem \ref{classification of subbundles} under the assumption that $\Ecal$ and $\Fcal$ have no common slopes. 
\end{prop}
\begin{proof}
Let $\Ecal$ and $\Fcal$ be vector bundles on $\schff$ which satisfy the condition \ref{rank inequalities for subbundles} in Theorem \ref{classification of subbundles}. From HN decompositions, we find decompositions
\begin{equation}\label{max common slope decomps} 
\Ecal \simeq \Ucal \oplus \nocommonslopered{\Ecal} \quad\quad\text{ and } \quad\quad \Fcal \simeq \Ucal \oplus \nocommonslopered{\Fcal}
\end{equation}
where $\nocommonslopered{\Ecal}$ and $\nocommonslopered{\Fcal}$ have no common slopes. For every $\mu \in \Q$ the decompositions \eqref{max common slope decomps} yield
\begin{align*}
\rk(\Ecal^{\geq \mu}) = \rk(\Ucal^{\geq \mu}) + \rk(\nocommonslopered{\Ecal}^{\geq \mu}) \quad\quad \text{ and } \quad\quad \rk(\Fcal^{\geq \mu}) = \rk(\Ucal^{\geq \mu}) + \rk(\nocommonslopered{\Fcal}^{\geq \mu}).
\end{align*}
Since $\Ecal$ and $\Fcal$ satisfy the condition \ref{rank inequalities for subbundles} in Theorem \ref{classification of subbundles}, we consequently find
\[ \rk(\nocommonslopered{\Ecal}^{\geq \mu}) \leq \rk(\nocommonslopered{\Fcal}^{\geq \mu}) \quad\quad\text{ for every } \mu \in \Q.\]
Moreover, an injective map $\nocommonslopered{\Ecal} \inj \nocommonslopered{\Fcal}$ gives rise to an injective map $\Ecal \inj \Fcal$ by direct summing wih the identity map on $\Ucal$. Hence we may prove sufficiency of the condition \ref{rank inequalities for subbundles} in Theorem \ref{classification of subbundles} after replacing $\Ecal$ and $\Fcal$ by $\nocommonslopered{\Ecal}$ and $\nocommonslopered{\Fcal}$. We thus have the desired reduction as $\nocommonslopered{\Ecal}$ and $\nocommonslopered{\Fcal}$ have no common slopes. 
\end{proof}

We now consider equivalence of the two conditions in Theorem \ref{classification of subbundles}. For convenience, we define the condition \ref{slopewise dominance for subbundles} in Theorem \ref{classification of subbundles} as a separate notion.
\begin{defn}
Let $\Ecal$ and $\Fcal$ be vector bundles on $\schff$. We say that $\Fcal$ \emph{slopewise dominates} $\Ecal$ if the condition \ref{slopewise dominance for subbundles} in Theorem \ref{classification of subbundles} is satisfied. 
\end{defn}
This notion is originally introduced by the author in \cite{Hong_quotvb} where equivalence of the conditions \ref{rank inequalities for subbundles} and \ref{slopewise dominance for subbundles} in Theorem \ref{classification of subbundles} is proved in the following form:
\begin{prop}[{\cite[Lemma 4.2.2]{Hong_quotvb}}]\label{equivalence of two characterizations for subbundles}
Let $\Ecal$ and $\Fcal$ be vector bundles on $\schff$. Then we have $\rk(\Ecal^{\geq \mu}) \leq \rk(\Fcal^{\geq \mu})$ for every $\mu \in \Q$ if and only if $\Fcal$ slopewise dominates $\Ecal$. 
\end{prop}

This equivalence will be extremely useful to us since the notion of slopewise dominance has several implications that are not easy to deduce directly from its equivalent condition \ref{rank inequalities for subbundles} in Theorem \ref{classification of subbundles}.
\begin{lemma}[{\cite[Lemma 4.2.3, Lemma 4.2.4, and Lemma 4.2.5]{Hong_quotvb}}]\label{implications of slopewise dominance}
Let $\Ecal$ and $\Fcal$ be vector bundles on $\schff$ such that $\Fcal$ slopewise dominates $\Ecal$. 
\begin{enumerate}[label=(\arabic*)]
\item\label{nonnegative degree for slopewise dominant pairs} We have an inequality
\[\deg(\Fcal)^\nonneg \geq \deg(\Ecal)^\nonneg.\]

\item\label{existence of max common factor decomp} There exist decompositions
\begin{equation}\label{max common factor decomp}
 \Ecal \simeq \Dcal \oplus \Ecal' \quad\quad \text{ and } \quad\quad \Fcal \simeq \Dcal \oplus \Fcal'
\end{equation}
satisfying the following properties:

\begin{enumerate}[label=(\roman*)]
\item\label{slopewise dominance for complement part} $\Fcal'$ slopewise dominates $\Ecal'$. 
\smallskip

\item\label{inequality for max slopes of complement parts} If $\Ecal' \neq 0$, we have $\mumax(\Fcal')>\mumax(\Ecal')$. 
\smallskip

\item\label{ineqaulities for min slope of common factor} If $\Dcal \neq 0$ and $\Ecal' \neq 0$, we have $\mumin(\Dcal) \geq \mumax(\Fcal') > \mumax(\Ecal')$.
\end{enumerate}
\begin{figure}[H]
\begin{tikzpicture}	

		\coordinate (left) at (0, 0);
		\coordinate (q0) at (1,2.5);
		\coordinate (q1) at (4, 4.5);
		\coordinate (q2) at (6.5, 4.5);
		\coordinate (q3) at (9, 2.5);
		

		\coordinate (p0) at (q0);
		\coordinate (p1) at (2.5, 3.5);
		\coordinate (p2) at (5.5, 3);
		\coordinate (p3) at (7.5, 0.5);
				
		\draw[step=1cm,thick, color=red] (left) -- (p0) --  (p1);
		\draw[step=1cm,thick, color=blue] (p1) -- (q1) -- (q2) -- (q3);
		\draw[step=1cm,thick, color=green] (p1) -- (p2) -- (p3);
		
		\draw [fill] (q0) circle [radius=0.05];		
		\draw [fill] (q1) circle [radius=0.05];		
		\draw [fill] (q2) circle [radius=0.05];		
		\draw [fill] (q3) circle [radius=0.05];		
		\draw [fill] (left) circle [radius=0.05];
		
		\draw [fill] (p0) circle [radius=0.05];		
		\draw [fill] (p1) circle [radius=0.05];		
		\draw [fill] (p2) circle [radius=0.05];		
		\draw [fill] (p3) circle [radius=0.05];		
		

		
		\path (q3) ++(0.8, 0.05) node {$\HN(\Fcal)$};
		\path (p3) ++(0.8, 0.05) node {$\HN(\Ecal)$};
		\path (left) ++(-0.3, -0.05) node {$O$};

		\path (p0) ++(-0.5, -0.6) node {\color{red}$\Dcal$};
		\path (q2) ++(1.5, -0.7) node {\color{blue}$\Fcal'$};
		\path (p2) ++(1, -0.6) node {\color{green}$\Ecal'$};



\end{tikzpicture}
\caption{Illustration of the decompositions \eqref{max common factor decomp} in terms of HN polygons.}\label{illustration of max common factor decomp}
\end{figure}
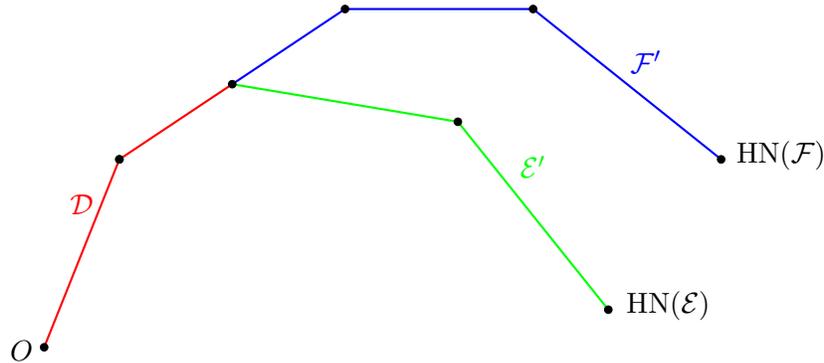
\smallskip

\item\label{duality of slopewise dominance for equal rank case} If $\rk(\Ecal) = \rk(\Fcal)$, then $\Ecal^\vee$ slopewise dominates $\Fcal^\vee$. 
\end{enumerate}
\end{lemma}

\begin{remark} The proof of \cite[Lemma 4.2.4]{Hong_quotvb} shows that the bundle $\Dcal$ in \eqref{max common factor decomp} represents the common part of $\HN(\Ecal)$ and $\HN(\Fcal)$, as illustrated in Figure \ref{illustration of max common factor decomp}. 
\end{remark}

\subsection{Formulation of the key inequality}\label{reformulation of statement}$ $

Thus far, by Propositions \ref{subbundles necessary condition}, \ref{reduction on common slopes} and \ref{equivalence of two characterizations for subbundles} we have reduced the proof of Theorem \ref{classification of subbundles} to establishing sufficiency of the condition \ref{slopewise dominance for subbundles} under the additional assumption that $\Ecal$ and $\Fcal$ have no common slopes. 
In this subsection, we will 
further reduce 
it to establishing the following quantitative statement:

\begin{prop}\label{key inequality}
Let $\Ecal$, $\Fcal$ and $\Qcal$ be vector bundles on $\schff$ with the following properties:
\begin{enumerate}[label=(\roman*)]
\item\label{slopewise dominance of F on E} $\Fcal$ slopewise dominates $\Ecal$. 
\smallskip

\item\label{dual slopewise dominance of E on Q} $\Ecal^\vee$ slopewise dominates $\Qcal^\vee$. 
\smallskip

\item\label{slopewise dominance of F on Q} $\Fcal$ slopewise dominates $\Qcal$. 
\smallskip




\item\label{no common slopes for E and F} $\Ecal$ and $\Fcal$ have no common slopes. 
\smallskip

\item\label{rank inequality for E and Q} $\rk(\Qcal) < \rk(\Ecal)$. 
\end{enumerate}
Then we have an inequality
\begin{equation}\label{deg inequality for inj}
\deg(\Ecal^\vee \otimes \Qcal)^\nonneg + \deg(\Qcal^\vee \otimes \Fcal)^\nonneg < \deg(\Ecal^\vee \otimes \Fcal)^\nonneg + \deg(\Qcal^\vee \otimes \Qcal)^\nonneg. 
\end{equation}
\end{prop}

\begin{remark}
This is the analogue of \cite[Proposition 4.3.5]{Hong_quotvb} in our situation. 
In fact, the statement of Proposition \ref{key inequality} and the statement of \cite[Proposition 4.3.5]{Hong_quotvb} have several notable similarities as follows:
\begin{enumerate}[label=(\arabic*)]
\item The inequalities considered in both statements are almost identical. 
\smallskip

\item By Proposition \ref{equivalence of two characterizations for subbundles}, the conditions \ref{dual slopewise dominance of E on Q} and \ref{slopewise dominance of F on Q} in Proposition \ref{key inequality} are almost equivalent to the corresponding conditions (ii) and (iii) in \cite[Proposition 4.3.5]{Hong_quotvb}
\smallskip

\item Both statements have the condition \ref{no common slopes for E and F} in common. 
\smallskip

\item The condition \ref{slopewise dominance of F on E} in each statement is precisely the condition on $\Ecal$ and $\Fcal$ in the corresponding main theorem in each context.
\end{enumerate}


On the other hand, the statement of Proposition \ref{key inequality} contains some changes from the statement of \cite[Proposition 4.3.5]{Hong_quotvb} as follows:
\begin{enumerate}[label=(\arabic*)]
\item The condition \ref{dual slopewise dominance of E on Q} in Proposition \ref{key inequality} does not have an additional ``equality condition" that appears in the condition (ii) in \cite[Proposition 4.3.5]{Hong_quotvb}.
\smallskip

\item\label{additional condition on key inequality} Proposition \ref{key inequality} has an additional condition \ref{rank inequality for E and Q} which is not present in \cite[Proposition 4.3.5]{Hong_quotvb}. 
\smallskip

\item The inequality \eqref{deg inequality for inj} in Proposition \ref{key inequality} is strict whereas the inequality considered in \cite[Proposition 4.3.5]{Hong_quotvb} is not. 
\end{enumerate}
Here the essential feature is the additional condition \ref{rank inequality for E and Q}, as the other two features are consequences of this feature. 

It is relatively easy to see why the strictness of \eqref{deg inequality for inj} is a consequence of the additional condition \ref{rank inequality for E and Q}. In fact, if we remove the condition \ref{rank inequality for E and Q} from Proposition \ref{key inequality}, then both sides of \eqref{deg inequality for inj} can be equal for many choices of $\Ecal, \Fcal$ and $\Qcal$. As an example, the reader can quickly check that both sides of \eqref{deg inequality for inj} are equal whenever $\Ecal = \Qcal$. There are also other choices, such as
\[\Ecal = \trivbundle, \quad\quad \Fcal = \trivbundle(1) \oplus \trivbundle(-1), \quad\quad \Qcal = \trivbundle(1),\]
for which both sides of \eqref{deg inequality for inj} are equal.

Let us now explain why the absence of the ``equality condition" in the condition \ref{dual slopewise dominance of E on Q} is a consequence of the condition \ref{rank inequality for E and Q}. As Lemma \ref{basic properties for candidate images} indicates, for our purpose we only need to establish the inequality \eqref{deg inequality for inj} when $\Qcal$ is a quotient of $\Ecal$ and a subbundle of $\Fcal$. The ``equality condition" in the condition (ii) of \cite[Proposition 4.3.5]{Hong_quotvb} is a necessary condition for $\Qcal$ to be a quotient of $\Ecal$. In our context, the condition \ref{rank inequality for E and Q} replaces this equality condition as a necessary condition for $\Qcal$ to be a quotient of $\Ecal$ with $\Qcal \neq \Ecal$. In fact, if we added the ``equality condition" to the condition \ref{dual slopewise dominance of E on Q} in Proposition \ref{key inequality}, the case $\rk(\Qcal) = \rk(\Ecal)$ would degenerate to the case $\Qcal = \Ecal$. 

Our discussion in the previous two paragraphs suggests that it is possible to state Proposition \ref{key inequality} without having all these new features by just adding the ``equality condition" to the condition \ref{dual slopewise dominance of E on Q}. However, we still want to have these features, as these features will notably simplify a number of our reduction arguments in the proof of Proposition \ref{key inequality}. 

\end{remark}

For the desired reduction, we need the following dual counterpart of Proposition \ref{subbundles necessary condition} and Proposition \ref{equivalence of two characterizations for subbundles}.

\begin{prop}\label{quotient bundles necessary condition}
Let $\Ecal$ be a vector bundle on $\schff$, and let $\Qcal$ be a quotient bundle of $\Ecal$. Then $\Ecal^\vee$ slopewise dominates $\Qcal^\vee$. 
\end{prop}

\begin{proof}
Since $\Qcal$ is a quotient bundle of $\Ecal$, its dual bundle $\Qcal^\vee$ is a subbundle of $\Ecal^\vee$. We thus have slopewise dominance of $\Ecal^\vee$ on $\Qcal^\vee$ by Proposition \ref{subbundles necessary condition} and Proposition \ref{equivalence of two characterizations for subbundles}.
\end{proof}

The following lemma relates some conditions in Proposition \ref{key inequality} to the construction that we introduced in Definition \ref{Hom with specified image}. 

\begin{lemma}\label{basic properties for candidate images}
Let $\Ecal, \Fcal$ and $\Qcal$ be vector bundles on $\schff$ such that the diamond $\Hom(\Ecal, \Fcal)_\Qcal$ (defined in Definition \ref{Hom with specified image}) is nonempty. 
\begin{enumerate}[label=(\arabic*)]
\item\label{Q as a quotient and a subbundle} $\Qcal$ is a quotient bundle of $\Ecal$ and a subbundle of $\Fcal$. 

\item\label{slopewise dominance conditions for Q} The bundles $\Ecal, \Fcal$ and $\Qcal$ satisfy the conditions \ref{dual slopewise dominance of E on Q} and \ref{slopewise dominance of F on Q} in Proposition \ref{key inequality}. 
\smallskip

\item\label{equal rank condition for Q} If $\Qcal \not\simeq \Ecal$, then $\rk(\Qcal) < \rk(\Ecal)$. 
\end{enumerate}
\end{lemma}

\begin{proof}
By definition, $\Hom(\Ecal, \Fcal)_\Qcal$ is the image of the map of diamonds
\[\Surj(\Ecal,\Qcal) \times_{\Spd\,\algclosedperfdfield} \Inj(\Qcal,\Fcal) \to \Hom(\Ecal,\Fcal)\]
induced by composition of bundle maps. Nonemptiness of $\Hom(\Ecal, \Fcal)_\Qcal$ therefore implies that both $\Surj(\Ecal, \Fcal)$ and $\Inj(\Ecal, \Fcal)$ are nonempty. By Proposition \ref{moduli of bundle maps fund facts}, we find that both $\Surj(\Ecal, \Qcal)$ and $\Inj(\Ecal, \Qcal)$ have $\algclosedperfdfield$-points, which precisely amounts to existence of a surjective bundle map $\Ecal \surj \Qcal$ and an injective bundle map $\Qcal \inj \Fcal$ as asserted in \ref{Q as a quotient and a subbundle}. Furthermore, we deduce \ref{slopewise dominance conditions for Q} from \ref{Q as a quotient and a subbundle} by Propositions \ref{subbundles necessary condition}, \ref{equivalence of two characterizations for subbundles}, and \ref{quotient bundles necessary condition}. 

Let us now assume that $\Qcal \not\simeq \Ecal$. As we already saw in the preceding paragraph, there exists a surjective map $\Ecal \surj \Qcal$. Its kernel $\Kcal$ is not trivial since the map is not an isomorphism by our assumption. We thus find
\[ \rk(\Qcal) = \rk(\Ecal) - \rk(\Kcal) < \rk(\Ecal),\]
thereby establishing \ref{equal rank condition for Q}. 
\end{proof}

With Lemma \ref{basic properties for candidate images}, we can explain why establishing Proposition \ref{key inequality} finishes the proof of Theorem \ref{classification of subbundles}. 

\begin{prop}\label{reduction to key inequality}
Proposition \ref{key inequality} implies sufficiency of the condition \ref{slopewise dominance for subbundles} in Theorem \ref{classification of subbundles} under the additional assumption that $\Ecal$ and $\Fcal$ have no common slopes. 
\end{prop}

\begin{proof}
Let $\Ecal$ and $\Fcal$ be vector bundles on $\schff$ with no common slopes such that $\Fcal$ slopewise dominates $\Ecal$. 
Let $S$ be the set of (isomorphism classes of) vector bundles $\Qcal$ on $\schff$ such that $\Hom(\Ecal, \Fcal)_\Qcal$ is nonempty. 
We wish to prove that $\Ecal$ is a subbundle of $\Fcal$, assuming Proposition \ref{key inequality}. 
By Lemma \ref{basic properties for candidate images}, it is enough to show $\Ecal \in S$.

Suppose for contradiction that $\Ecal \notin S$. By Proposition \ref{key inequality} and Lemma \ref{basic properties for candidate images}, every $\Qcal \in S$ should satisfy the strict inequality
\[\deg(\Ecal^\vee \otimes \Qcal)^\nonneg + \deg(\Qcal^\vee \otimes \Fcal)^\nonneg < \deg(\Ecal^\vee \otimes \Fcal)^\nonneg + \deg(\Qcal^\vee \otimes \Qcal)^\nonneg.\]
Now the dimension formulas in Proposition \ref{moduli of bundle maps fund facts} and Proposition \ref{properties of Hom with specified image} imply that for every $\Qcal \in S$ we have
\begin{equation}\label{strict dimension inequality for Hom strata} 
\dim |\Hom(\Ecal, \Fcal)_\Qcal| < \dim \Hom(\Ecal, \Fcal).
\end{equation}
On the other hand, we have a decomposition
\[|\Hom(\Ecal, \Fcal)| = \coprod_{\Qcal \in S}|\Hom(\Ecal, \Fcal)_\Qcal|.\]
We thus use Proposition \ref{properties of Hom with specified image} and \eqref{strict dimension inequality for Hom strata} to find
\[\dim |\Hom(\Ecal, \Fcal)| = \sup_{\Qcal \in S} \dim |\Hom(\Ecal, \Fcal)_\Qcal| < \dim |\Hom(\Ecal, \Fcal)|,\]
thereby obtaining the desired contradiction. 
\end{proof}

\subsection{Reduction on slopes and ranks}\label{reduction of ranks}$ $

Our goal for the rest of this paper is to establish Proposition \ref{key inequality}. For our convenience, we introduce the following notation: 
\begin{defn}\label{definition of cEF(Q)}
For arbitrary vector bundles $\Ecal, \Fcal$ and $\Qcal$ on $\schff$, we define
\[c_{\Ecal, \Fcal}(\Qcal) := \deg(\Ecal^\vee \otimes \Fcal)^\nonneg + \deg(\Qcal^\vee \otimes \Qcal)^\nonneg - \deg(\Ecal^\vee \otimes \Qcal)^\nonneg -\deg(\Qcal^\vee \otimes \Fcal)^\nonneg.\]
Note that the inequality \eqref{deg inequality for inj} in Proposition \ref{key inequality} can be written as $c_{\Ecal, \Fcal}(\Qcal) > 0$.
\end{defn}

In this subsection, we reduce the proof of Proposition \ref{key inequality} to the case where the following additional conditions are satisfied:
\begin{enumerate}[label=(\roman*), start=5]
\item[\refstepcounter{enumi}\theenumi'] $\rk(\Qcal) = \rk(\Ecal) - 1$. 
\smallskip

\item all slopes of $\Ecal, \Fcal$ and $\Qcal$ are integers. 

\smallskip

\item $\mumax(\Ecal) = 0$. 
\end{enumerate}
The following lemma will be crucial for this task. 

\begin{lemma}[{\cite[Lemma 3.2.7 and Lemma 3.2.8]{Hong_quotvb}}]\label{degree computing lemmas}
Let $\Vcal$ and $\Wcal$ be arbitrary vector bundles on $\schff$. 
\begin{enumerate}[label = (\arabic*)]
\item For vector bundles $\vertstretch{\Vcal}$ and $\vertstretch{\Wcal}$ on $\schff$ whose HN polygons are obtained by vertically stretching $\HN(\Vcal)$ and $\HN(\Wcal)$ by a positive integer factor $C$, we have
\[ \deg(\vertstretch{\Vcal}^\vee \otimes \vertstretch{\Wcal})^\nonneg = C \cdot \deg(\Vcal^\vee \otimes \Wcal)^\nonneg.\]

\item For vector bundles $\Vcal(\lambda) := \Vcal \otimes \trivbundle(\lambda)$ and $\Wcal(\lambda) := \Vcal \otimes \trivbundle(\lambda)$, we have
\[\deg(\Vcal(\lambda)^\vee \otimes \Wcal(\lambda))^\nonneg = \rk(\trivbundle(\lambda))^2 \cdot \deg(\Vcal^\vee \otimes \Wcal)^\nonneg.\]
\end{enumerate}
\end{lemma}

Let us now carry out the proposed reduction. 

\begin{prop}\label{key inequality reduction to integer slopes}
We may prove Proposition \ref{key inequality} under the assumption that all slopes of $\Ecal$, $\Fcal$ and $\Qcal$ are integers. 
\end{prop}

\begin{proof}
Let $\Ecal, \Fcal$ and $\Qcal$ be as in the statement of Proposition \ref{key inequality}. Take $C$ to be a common multiple of all denominators of the slopes in $\HN(\Ecal), \HN(\Fcal)$ and $\HN(\Qcal)$, and define $\vertstretch{\Ecal}, \vertstretch{\Fcal}$ and $\vertstretch{\Qcal}$ to be vector bundles on $\schff$ whose HN polygons are obtained by vertically stretching $\HN(\Ecal), \HN(\Fcal)$ and $\HN(\Qcal)$ by a factor $C$. Then we have the following facts:
\begin{enumerate}[label=(\arabic*)]
\item\label{slopes after stretch} All slopes of $\vertstretch{\Ecal}, \vertstretch{\Fcal}$ and $\vertstretch{\Qcal}$ are integers. 
\smallskip

\item\label{slopewise dominance after stretch} The conditions \ref{slopewise dominance of F on E} - \ref{no common slopes for E and F} in Proposition \ref{key inequality} are satisfied after replacing $\Ecal, \Fcal$ and $\Qcal$ by $\vertstretch{\Ecal}, \vertstretch{\Fcal}$ and $\vertstretch{\Qcal}$. 

\item\label{rank after stretch} $\rk(\Qcal) = \rk(\vertstretch{\Qcal})$ and $\rk(\Ecal) = \rk(\vertstretch{\Ecal})$. 
\smallskip

\item\label{cEF(Q) after stretch} $c_{\vertstretch{\Ecal}, \vertstretch{\Fcal}}(\vertstretch{\Qcal}) = C \cdot c_{\Ecal, \Fcal}(\Qcal)$. 
\end{enumerate}
Indeed, \ref{slopes after stretch}, \ref{slopewise dominance after stretch} and \ref{rank after stretch} are evident by construction while \ref{cEF(Q) after stretch} follows from Lemma \ref{degree computing lemmas}. Now \ref{slopewise dominance after stretch}, \ref{rank after stretch} and \ref{cEF(Q) after stretch} together imply that we may prove Proposition \ref{key inequality} after replacing $\Ecal, \Fcal$ and $\Qcal$ by $\vertstretch{\Ecal}, \vertstretch{\Fcal}$ and $\vertstretch{\Qcal}$, thereby yielding the desired reduction by \ref{slopes after stretch}. 
\end{proof}

\begin{prop}\label{key inequality reduction to minimal rank}
We may prove Proposition \ref{key inequality} under the following additional conditions:
\begin{enumerate}[label=(\roman*), start=5]
\item[\refstepcounter{enumi}\customlabel{minimal rank for E in step 2}{\theenumi'}] $\rk(\Qcal) = \rk(\Ecal) - 1$. 
\smallskip

\item\label{integer slopes of E, F, Q in step 2} all slopes of $\Ecal, \Fcal$ and $\Qcal$ are integers. 
\end{enumerate}
\end{prop}

\begin{proof}
Suppose that Proposition \ref{key inequality} holds when the conditions \ref{minimal rank for E in step 2} and \ref{integer slopes of E, F, Q in step 2} are satisfied. We wish to deduce the general case of Proposition \ref{key inequality} from this assumption. In light of Proposition \ref{key inequality reduction to integer slopes}, we assume that the condition \ref{integer slopes of E, F, Q in step 2} is satisfied. Under this assumption, we proceed by induction on $\rk(\Ecal) - \rk(\Qcal)$. Since the base case $\rk(\Ecal) - \rk(\Qcal) = 1$ follows from our assumption, we only need to consider the induction step. 


We first reduce our induction step to the case $\mumax(\Ecal) = 0$. For this, we take $\lambda := \mumax(\Ecal)$ and consider the vector bundles 
\[ \Ecal(-\lambda):= \Ecal \otimes \trivbundle(-\lambda), \quad\quad \Fcal(-\lambda):= \Fcal \otimes \trivbundle(-\lambda), \quad\quad \Qcal(-\lambda):= \Qcal \otimes \trivbundle(-\lambda).\]
Note that $\lambda = \mumax(\Ecal)$ is an integer by the condition \ref{integer slopes of E, F, Q in step 2} that we assumed. In particular, the bundle $\trivbundle(\lambda)$ has rank $1$ by Lemma \ref{rank, degree and dual of stable bundles}. It is therefore straightforward to check the following identity using Definition \ref{o-r-over-s}. 
\[ \trivbundle(\mu) \otimes \trivbundle(-\lambda) = \trivbundle(\mu - \lambda) \quad\quad\quad \text{ for all } \mu \in \Q.\]
Then by HN decompositions we observe that $\HN(\Ecal(-\lambda)), \HN(\Fcal(-\lambda))$ and $\HN(\Qcal(-\lambda))$ are obtained by reducing all slopes of $\HN(\Ecal), \HN(\Fcal)$ and $\HN(\Qcal)$ by $\lambda$. Consequently, we deduce the following facts:

\begin{enumerate}[label=(\arabic*)]
\item\label{max slope after shear} $\mumax(\Ecal(-\lambda)) = \mumax(\Ecal) - \lambda = 0$. 
\smallskip

\item\label{slopewise dominance after shear} The conditions \ref{slopewise dominance of F on E} - \ref{no common slopes for E and F} in Proposition \ref{key inequality} and the additional condition \ref{integer slopes of E, F, Q in step 2} are satisfied after replacing $\Ecal, \Fcal$ and $\Qcal$ by $\Ecal(-\lambda), \Fcal(-\lambda)$ and $\Qcal(-\lambda)$.  
\smallskip

\item\label{rank after shear} $\rk(\Qcal(-\lambda)) = \rk(\Qcal)$ and $\rk(\Ecal(-\lambda)) = \rk(\Ecal)$. 
\end{enumerate}
Moreover, by Lemma \ref{degree computing lemmas} we get an identity
\begin{equation}\label{cEF(Q) after shear}
c_{\Ecal(-\lambda), \Fcal(-\lambda)}(\Qcal(-\lambda)) = c_{\Ecal, \Fcal}(\Qcal)
\end{equation}
since $\rk(\trivbundle(-\lambda)) = 1$ as already noted. Now \ref{slopewise dominance after shear}, \ref{rank after shear} and \eqref{cEF(Q) after shear} together imply that we may replace $\Ecal, \Fcal$ and $\Qcal$ by $\Ecal(-\lambda), \Fcal(-\lambda)$ and $\Qcal(-\lambda)$ for the induction step, thereby yielding the desired reduction by \ref{max slope after shear}. 

Let us now assume that $\mumax(\Ecal) = 0$. For our induction step we assume $\rk(\Ecal) - \rk(\Qcal) >1$, or equivalently $\rk(\Ecal) > \rk(\Qcal) +1$. Then we can write 
\begin{equation}\label{reduction on rank induction step decomposition}
\Ecal = \rankred{\Ecal} \oplus \trivbundle 
\end{equation}
where $\mumax(\rankred{\Ecal}) \leq 0$ and $\rk(\rankred{\Ecal})  > \rk(\Qcal)$.

Our next assertion is that the conditions \ref{slopewise dominance of F on E} - \ref{no common slopes for E and F} in Proposition \ref{key inequality} and the additional condition \ref{integer slopes of E, F, Q in step 2} are satisfied after replacing $\Ecal$ by $\rankred{\Ecal}$. The condition \ref{slopewise dominance of F on Q} is trivial since $\Fcal$ and $\Qcal$ remain unchanged. The condition \ref{no common slopes for E and F} and the additional condition \ref{integer slopes of E, F, Q in step 2} are also obvious by construction. For the condition \ref{slopewise dominance of F on E}, we need to check slopewise dominance of $\Fcal$ on $\rankred{\Ecal}$, which follows by combining slopewise dominance of $\Fcal$ on $\Ecal$ and slopewise dominance of $\Ecal$ on $\rankred{\Ecal}$; in fact, the former is given by the condition \ref{slopewise dominance of F on E} for $\Ecal$ and $\Fcal$, whereas the latter follows by applying Proposition \ref{subbundles necessary condition} and Proposition \ref{equivalence of two characterizations for subbundles} to the observation that $\rankred{\Ecal}$ is a subbundle of $\Ecal$ by \eqref{reduction on rank induction step decomposition}. For the remaining condition \ref{dual slopewise dominance of E on Q}, we need to show slopewise dominance of $\rankred{\Ecal}^\vee$ on $\Qcal^\vee$. From \eqref{reduction on rank induction step decomposition} we obtain
\begin{equation}\label{reduction on rank induction step dual decomposition}
\Ecal^\vee = \trivbundle \oplus \rankred{\Ecal}^\vee.
\end{equation}
Moreover, by Lemma \ref{rank, degree and dual of stable bundles}, we have $\mumin(\Ecal^\vee) = -\mumax(\Ecal) = 0$ and $\mumin(\rankred{\Ecal}^\vee) = -\mumax(\rankred{\Ecal})$. Hence we see that $\HN(\rankred{\Ecal}^\vee)$ is obtained from $\HN(\Ecal^\vee)$ by removing the line segment over the interval $(\rk(\Ecal)-1, \rk(\Ecal)]$, as indicated in Figure \ref{reduction on rank dual polygons}. Since $\rk(\Ecal) > \rk(\Qcal)$ by our assumption, this removal process does not affect slopewise dominance on $\Qcal^\vee$. In other words, slopewise dominance of $\Ecal^\vee$ on $\Qcal^\vee$ as given in the condition \ref{dual slopewise dominance of E on Q} implies slopewise dominance of $\rankred{\Ecal}^\vee$ on $\Qcal^\vee$ as desired.

\begin{figure}[H]
\begin{tikzpicture}	

		\coordinate (left) at (0, 0);
		\coordinate (q0) at (1,2);
		\coordinate (q1) at (2, 3);
		\coordinate (q2) at (3.5, 3.5);
		\coordinate (q3) at (5, 3.5);
		

		\coordinate (p0) at (1.5, 1);
		\coordinate (p1) at (3, 1.3);
		\coordinate (p2) at (4, 0.7);
				
		\draw[step=1cm,thick] (left) -- (q0) --  (q1) -- (q2) -- (q3);
		\draw[step=1cm,thick] (left) -- (p0) --  (p1) -- (p2);
		
		\draw [fill] (q0) circle [radius=0.05];		
		\draw [fill] (q1) circle [radius=0.05];		
		\draw [fill] (q2) circle [radius=0.05];		
		\draw [fill] (q3) circle [radius=0.05];		
		\draw [fill] (left) circle [radius=0.05];
		
		\draw [fill] (p0) circle [radius=0.05];		
		\draw [fill] (p1) circle [radius=0.05];		
		\draw [fill] (p2) circle [radius=0.05];		

		\draw[step=1cm,dotted] (4.5, -0.4) -- (4.5, 3.6);

		\node at (4.4,-0.8) {\scriptsize $\rk(\Ecal)-1$};
		
		\path (q3) ++(0.2, 0.3) node {$\HN(\Ecal^\vee)$};
		\path (p2) ++(-0.2, -0.3) node {$\HN(\Qcal^\vee)$};
		\path (left) ++(-0.3, -0.05) node {$O$};

\end{tikzpicture}
\begin{tikzpicture}[scale=0.4]
        \pgfmathsetmacro{\textycoordinate}{8}
		\draw[->, line width=0.6pt] (0, \textycoordinate) -- (1.5,\textycoordinate);
		\draw (0,0) circle [radius=0.00];	
        \hspace{0.2cm}
\end{tikzpicture}
\begin{tikzpicture}	

		\coordinate (left) at (0, 0);
		\coordinate (q0) at (1,2);
		\coordinate (q1) at (2, 3);
		\coordinate (q2) at (3.5, 3.5);
		\coordinate (q3) at (4.5, 3.5);
		

		\coordinate (p0) at (1.5, 1);
		\coordinate (p1) at (3, 1.3);
		\coordinate (p2) at (4, 0.7);
				
		\draw[step=1cm,thick] (left) -- (q0) --  (q1) -- (q2) -- (q3);
		\draw[step=1cm,thick] (left) -- (p0) --  (p1) -- (p2);
		
		\draw [fill] (q0) circle [radius=0.05];		
		\draw [fill] (q1) circle [radius=0.05];		
		\draw [fill] (q2) circle [radius=0.05];		
		\draw [fill] (q3) circle [radius=0.05];		
		\draw [fill] (left) circle [radius=0.05];
		
		\draw [fill] (p0) circle [radius=0.05];		
		\draw [fill] (p1) circle [radius=0.05];		
		\draw [fill] (p2) circle [radius=0.05];		
		
		\draw[step=1cm,dotted] (4.5, -0.4) -- (4.5, 3.6);

		\draw[step=1cm,dashed] (q3) -- (5, 3.5);

		\node at (4.4,-0.8) {\scriptsize $\rk(\Ecal)-1$};
		
		\path (q3) ++(0.2, 0.3) node {$\HN(\rankred{\Ecal^\vee})$};
		\path (p2) ++(-0.2, -0.3) node {$\HN(\Qcal^\vee)$};
		\path (left) ++(-0.3, -0.05) node {$O$};

\end{tikzpicture}
\caption{Illustration of the induction step in terms of dual HN polygons}\label{reduction on rank dual polygons}
\end{figure}
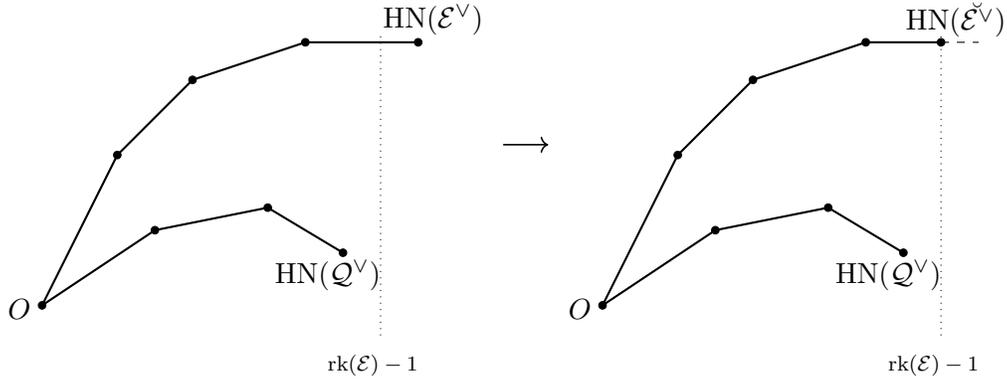

Now, as $\rk(\rankred{\Ecal}) - \rk(\Qcal) < \rk(\Ecal) - \rk(\Qcal)$, our discussion in the preceding paragraph shows that we may apply the induction hypothesis to deduce
\begin{equation}\label{step 1 codim inequality for E', F, Q}
c_{\rankred{\Ecal}, \Fcal}(\Qcal) > 0
\end{equation}
For the desired inequality $c_{\Ecal, \Fcal}(\Qcal) > 0$, we use the decomposition \eqref{reduction on rank induction step dual decomposition} to compute
\begin{align*}
\deg(\Ecal^\vee \otimes \Fcal)^\nonneg &= \deg((\rankred{\Ecal}^\vee \oplus \trivbundle) \otimes \Fcal)^\nonneg\\
&= \deg(\rankred{\Ecal}^\vee \otimes \Fcal)^\nonneg + \deg(\trivbundle \otimes \Fcal)^\nonneg\\
&= \deg(\rankred{\Ecal}^\vee \otimes \Fcal)^\nonneg + \deg(\Fcal)^\nonneg,\\
\deg(\Ecal^\vee \otimes \Qcal)^\nonneg &= \deg((\rankred{\Ecal}^\vee \oplus \trivbundle) \otimes \Qcal)^\nonneg\\
&= \deg(\rankred{\Ecal}^\vee \otimes \Qcal)^\nonneg + \deg(\trivbundle \otimes \Qcal)^\nonneg\\
&= \deg(\rankred{\Ecal}^\vee \otimes \Qcal)^\nonneg + \deg(\Qcal)^\nonneg.
\end{align*}
Then by Definition \ref{definition of cEF(Q)} we find
\[c_{\Ecal, \Fcal}(\Qcal) = c_{\rankred{\Ecal}, \Fcal}(\Qcal) + \deg(\Fcal)^\nonneg - \deg(\Qcal)^\nonneg.\]
Since $\Fcal$ slopewise dominates $\Qcal$ by the condition \ref{slopewise dominance of F on Q}, we use Lemma \ref{implications of slopewise dominance}
to find
\begin{equation}\label{step 1 inequality for codim E, F, Q and codim E', F, Q}
c_{\Ecal, \Fcal}(\Qcal) \geq c_{\rankred{\Ecal}, \Fcal}(\Qcal).
\end{equation}
We thus deduce the desired inequality $c_{\Ecal, \Fcal}(\Qcal)>0$ from \eqref{step 1 codim inequality for E', F, Q} and \eqref{step 1 inequality for codim E, F, Q and codim E', F, Q}. 
\end{proof}

\begin{remark}
Proposition \ref{key inequality reduction to integer slopes} and Proposition \ref{key inequality reduction to minimal rank} are the counterparts of \cite[Proposition 4.4.5 and Proposition 4.4.6]{Hong_quotvb} in our setting. Naturally, their proofs closely follow the proofs of their counterparts. 

Here we note a notable difference between Proposition \ref{key inequality reduction to minimal rank} and its counterpart \cite[Proposition 4.4.6]{Hong_quotvb}. In Proposition \ref{key inequality reduction to minimal rank}, our reduction does not reach the case $\rk(\Qcal) = \rk(\Ecal)$; on the other hand, the reduction in \cite[Proposition 4.4.6]{Hong_quotvb} reaches the case $\rk(\Qcal) = \rk(\Fcal)$. We will see that our argument in \S\ref{degeneration process on dual bundles} requires some additional work because of this difference. 

At first glance, this difference seems to be a direct consequence of the condition \ref{rank inequality for E and Q} in Proposition \ref{key inequality}. However, even if we remove this condition from Proposition \ref{key inequality}, we are still unable to reach the case $\rk(\Qcal) = \rk(\Ecal)$ by our reduction argument in Proposition \ref{key inequality reduction to minimal rank}. The main issue is that, as remarked in \S\ref{reformulation of statement}, removing the condition \ref{rank inequality for E and Q} from Proposition \ref{key inequality} makes the inequality \eqref{deg inequality for inj} a nonstrict inequality where equality may hold even if $\Qcal \not\simeq \Ecal$. In fact, in the proof of \cite[Proposition 4.4.6]{Hong_quotvb} the reduction to the case $\rk(\Qcal) = \rk(\Fcal)$ crucially uses the equality condition $\Qcal \simeq \Fcal$ for the inequality in \cite[Proposition 4.3.5]{Hong_quotvb}.

We also point out that our proof of Proposition \ref{key inequality reduction to integer slopes} and Proposition \ref{key inequality reduction to minimal rank} enjoys the benefits from several features of Proposition \ref{key inequality} as remarked in \S\ref{reformulation of statement}. For example, when we replace the triple $(\Ecal, \Fcal, \Qcal)$ by another triple, such as $(\vertstretch{\Ecal}, \vertstretch{\Fcal}, \vertstretch{\Qcal})$ in the proof of Proposition \ref{key inequality reduction to integer slopes} or $(\Ecal(-\lambda), \Fcal(-\lambda), \Qcal(-\lambda))$ in the proof of Proposition \ref{key inequality reduction to minimal rank}, it is straightforward to check the condition \ref{dual slopewise dominance of E on Q} in Proposition \ref{key inequality} for the new triple because of the absence of the equality condition.

\end{remark}

\begin{prop}\label{key inequality reduction to zero max slope}
We may prove Proposition \ref{key inequality} under the following additional conditions:
\begin{enumerate}[label=(\roman*), start=5]
\item[\refstepcounter{enumi}\customlabel{minimal rank for E in step 3}{\theenumi'}] $\rk(\Qcal) = \rk(\Ecal) - 1$. 
\smallskip

\item\label{integer slopes of E, F, Q in step 3} all slopes of $\Ecal, \Fcal$ and $\Qcal$ are integers. 
\smallskip

\item\label{zero max slope for E in step 3} $\mumax(\Ecal) = 0$. 
\end{enumerate}
\end{prop}

\begin{proof}
Let $\Ecal, \Fcal$ and $\Qcal$ be vector bundles on $\schff$ which satisfy the conditions \ref{minimal rank for E in step 3} and \ref{integer slopes of E, F, Q in step 3} in addition to all conditions in Proposition \ref{key inequality}. By Proposition \ref{key inequality reduction to minimal rank}, it suffices to consider such vector bundles for the proof of Proposition \ref{key inequality}. For the desired reduction, we can argue exactly as in the second paragraph of the proof of Proposition \ref{key inequality reduction to minimal rank}; in other words, we set $\lambda := \mumax(\Ecal)$ and replace $\Ecal, \Fcal$ and $\Qcal$ by
\[ \Ecal(-\lambda):= \Ecal \otimes \trivbundle(-\lambda), \quad\quad \Fcal(-\lambda):= \Fcal \otimes \trivbundle(-\lambda), \quad\quad \Qcal(-\lambda):= \Qcal \otimes \trivbundle(-\lambda)\]
to obtain the desired reduction. 
\end{proof}

\subsection{Degeneration of the dual bundles}\label{degeneration process on dual bundles}$ $

By Proposition \ref{key inequality reduction to zero max slope}, our remaining goal is to prove the following statement:

\begin{prop}\label{reduced key inequality}
Let $\Ecal$, $\Fcal$ and $\Qcal$ be vector bundles on $\schff$ with the following properties:
\begin{enumerate}[label=(\roman*)]
\item\label{slopewise dominance of F on E, reduced} $\Fcal$ slopewise dominates $\Ecal$. 
\smallskip

\item\label{dual slopewise dominance of E on Q, reduced} $\Ecal^\vee$ slopewise dominates $\Qcal^\vee$. 
\smallskip

\item\label{slopewise dominance of F on Q, reduced} $\Fcal$ slopewise dominates $\Qcal$. 
\smallskip




\item\label{no common slopes for E and F, reduced} $\Ecal$ and $\Fcal$ have no common slopes. 
\smallskip

\item\label{minimal rank for E, reduced} $\rk(\Qcal) = \rk(\Ecal) - 1$. 
\smallskip

\item\label{integer slopes of E, F, Q, reduced} all slopes of $\Ecal, \Fcal$ and $\Qcal$ are integers. 
\smallskip

\item\label{zero max slope for E, reduced} $\mumax(\Ecal) = 0$. 
\end{enumerate}
Then we have an inequality
\begin{equation}\label{deg inequality for inj, reduced}
c_{\Ecal, \Fcal}(\Qcal)>0.
\end{equation}
\end{prop}
For the rest of this paper, we fix vector bundles $\Ecal, \Fcal$ and $\Qcal$ as in the statement of Proposition \ref{reduced key inequality}. 

Let us briefly sketch our proof of Proposition \ref{reduced key inequality}. The key idea is to construct a finite sequence
\[ \Ecal = \Ecal_0, ~\Ecal_1, ~\cdots, ~\Ecal_r = \Qcal\]
which is ``dually degenerating" in the sense that $\Ecal_i^\vee$ slopewise dominates $\Ecal_{i+1}^\vee$ for each $i = 0, 1, \cdots, r$. By this ``degenerating" property, we will obtain 
\begin{equation}\label{decreasing cEF(Q) for degenerating sequence}
c_{\Ecal_i, \Fcal}(\Qcal) \geq c_{\Ecal_{i+1}, \Fcal}(\Qcal) \quad\quad\quad \text{ for each } i = 0, 1, \cdots, r-1.
\end{equation}
Consequently we will deduce
\begin{equation}\label{deg inequality from degenerating sequence} 
c_{\Ecal, \Fcal}(\Qcal) = c_{\Ecal_0, \Fcal}(\Qcal) \geq c_{\Ecal_r, \Fcal}(\Qcal) = c_{\Qcal, \Fcal}(\Qcal) = 0
\end{equation}
where the last identity follows immediately from Definition \ref{definition of cEF(Q)}. We will then show that equality in \eqref{deg inequality from degenerating sequence} never holds by examining the equality condition of the inequality \eqref{decreasing cEF(Q) for degenerating sequence}.

\begin{remark}
Our proof of Proposition \ref{reduced key inequality} will closely follow the argument in \cite[\S4.4]{Hong_quotvb}. However, there are some adjustments that we need to make. 

In \cite[\S4.4]{Hong_quotvb}, the construction of the degenerating sequence crucially relies on the condition $\rk(\Qcal) = \rk(\Fcal)$. In our context, since we begin with the condition $\rk(\Qcal) = \rk(\Ecal) - 1$, we will need an additional step to attain a similar ``equal rank" condition. We will thus construct $\Ecal_1$ by cutting down $\Ecal$ so that we have $\rk(\Qcal) = \rk(\Ecal_1)$. 


The main subtlety for our proof of Proposition \ref{reduced key inequality} lies in establishing nonstrictness of the inequality \eqref{deg inequality for inj, reduced}. In \cite[\S4.4]{Hong_quotvb}, the equality condition for the inequality $c_{\Ecal, \Fcal}(\Qcal) \geq 0$ is established by showing that $c_{\Ecal, \Fcal}(\Qcal)$ strictly decreases during the first step, or more precisely $c_{\Ecal, \Fcal}(\Qcal)>c_{\Ecal, \Fcal_1}(\Qcal)$. In our situation, we will have to simultaneously consider the first two steps because of the additional step that we described in the preceding paragraph. Our argument also requires some additional adjustments on details as we will see in the proof of Proposition \ref{strict decreasing cEF(Q) in first two steps}. 


\end{remark}

We now begin our proof of Proposition \ref{reduced key inequality}. As remarked above, the first step of our construction aims to attain an equal rank condition by cutting down $\Ecal$. 

\begin{prop}\label{construction of E1} Let $\Ecal_1$ be a direct summand of $\Ecal$ such that
\[ \Ecal = \Ecal_1 \oplus \trivbundle.\]
Then we have the following facts:
\begin{enumerate}[label=(\arabic*)]
\item\label{no common slopes for E1 and F} $\Ecal_1$ and $\Fcal$ have no common slopes. 
\smallskip

\item\label{rank of E1} $\rk(\Qcal) = \rk(\Ecal_1)$
\smallskip

\item\label{integer slopes for E1} all slopes of $\Ecal_1$ are integers. 
\smallskip

\item\label{max slope of E1} $\mumax(\Ecal_1) \leq 0$. 
\smallskip

\item\label{slopewise dominance of Q on E1} $\Qcal$ slopewise dominates $\Ecal_1$
\smallskip

\item\label{decreasing c for E1} We have an inequality
\[c_{\Ecal, \Fcal}(\Qcal) \geq c_{\Ecal_1, \Fcal}(\Qcal)\]
with equality if and only if $\deg(\Fcal)^\nonneg = \deg(\Qcal)^\nonneg$. 
\end{enumerate}
\end{prop}

\begin{proof}
By construction, the statements \ref{no common slopes for E1 and F}, \ref{rank of E1}, \ref{integer slopes for E1} and \ref{max slope of E1} follow immediately from the conditions \ref{no common slopes for E and F}, \ref{minimal rank for E, reduced}, \ref{integer slopes of E, F, Q, reduced} and \ref{zero max slope for E, reduced} in Proposition \ref{reduced key inequality}. In addition, the condition \ref{dual slopewise dominance of E on Q} and the condition \ref{zero max slope for E, reduced} in Proposition \ref{reduced key inequality} together yield slopewise dominance of $\Ecal_1^\vee$ on $\Qcal^\vee$, which consequently implies the statement \ref{slopewise dominance of Q on E1} by Lemma \ref{implications of slopewise dominance} and the statement \ref{rank of E1}.
Moreover, we can argue as in the last paragraph of the proof of Proposition \ref{key inequality reduction to minimal rank} to find
\[c_{\Ecal, \Fcal}(\Qcal) = c_{\Ecal_1, \Fcal}(\Qcal) + \deg(\Fcal)^\nonneg - \deg(\Qcal)^\nonneg,\]
from which the statement \ref{decreasing c for E1} follows by Lemma \ref{implications of slopewise dominance} and the condition \ref{slopewise dominance of F on Q, reduced} in Proposition \ref{reduced key inequality}. 
\end{proof}

In order to describe the rest of our construction, we recall the following notion from \cite{Hong_quotvb}:

\begin{defn}\label{max slope reduction definition}
Let $\Vcal$ and $\Wcal$ be nonzero vector bundles on $\schff$ with integer slopes such that $\Vcal$ slopewise dominates $\Wcal$. We refer to the vector bundle
\[\maxslopered{\Vcal} := \trivbundle(\mumax(\Wcal))^{\oplus \rk(\Vcal^{>\mumax(\Wcal)})} \oplus \Vcal^{\leq \mumax(\Wcal)}\]
as the \emph{maximal slope reduction} of $\Vcal$ to $\Wcal$. In other words, $\maxslopered{\Vcal}$ is the vector bundle on $\schff$ obtained from $\Vcal$ by reducing all slopes of $\Vcal^{>\mumax(\Wcal)}$ to $\mumax(\Wcal)$. 
\end{defn}

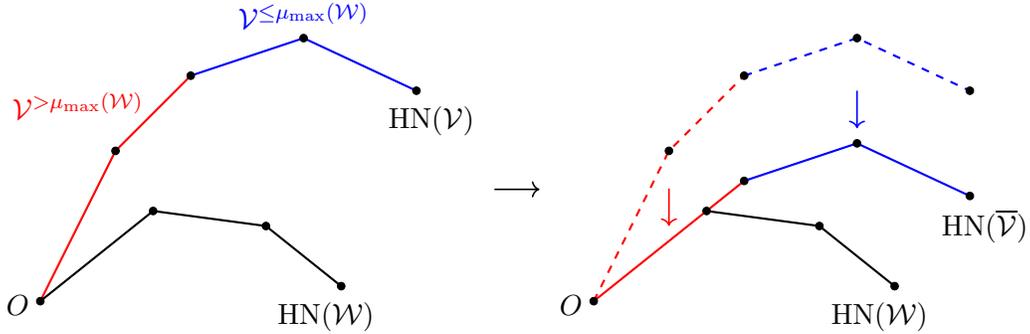
\begin{figure}[H]
\begin{tikzpicture}	

		\coordinate (left) at (0, 0);
		\coordinate (q0) at (1,2);
		\coordinate (q1) at (2, 3);
		\coordinate (q2) at (3.5, 3.5);
		\coordinate (q3) at (5, 2.8);
		

		\coordinate (p0) at (1.5, 1.2);
		\coordinate (p1) at (3, 1);
		\coordinate (p2) at (4, 0.2);
				
		\draw[step=1cm,thick, color=red] (left) -- (q0) --  (q1);
		\draw[step=1cm,thick, color=blue] (q1) -- (q2) -- (q3);
		\draw[step=1cm,thick, color=black] (left) -- (p0) -- (p1) -- (p2);

		\draw [fill] (q0) circle [radius=0.05];		
		\draw [fill] (q1) circle [radius=0.05];		
		\draw [fill] (q2) circle [radius=0.05];		
		\draw [fill] (q3) circle [radius=0.05];		
		\draw [fill] (left) circle [radius=0.05];
		
		\draw [fill] (p0) circle [radius=0.05];		
		\draw [fill] (p1) circle [radius=0.05];		
		\draw [fill] (p2) circle [radius=0.05];		


		
		\path (q3) ++(0.2, -0.4) node {$\HN(\Vcal)$};
		\path (p2) ++(-0.2, -0.4) node {$\HN(\Wcal)$};
		\path (left) ++(-0.3, -0.05) node {$O$};

		\path (q0) ++(-0.5, 0.6) node {\color{red}$\Vcal^{>\mumax(\Wcal)}$};
		\path (q2) ++(0, 0.3) node {\color{blue}$\Vcal^{\leq \mumax(\Wcal)}$};

\end{tikzpicture}
\begin{tikzpicture}[scale=0.4]
        \pgfmathsetmacro{\textycoordinate}{5}
		\draw[->, line width=0.6pt] (0, \textycoordinate) -- (1.5,\textycoordinate);
		\draw (0,0) circle [radius=0.00];	
        \hspace{0.2cm}
\end{tikzpicture}
\begin{tikzpicture}
		\pgfmathsetmacro{\reducedycoordinate}{1.6}
	

		\coordinate (left) at (0, 0);
		\coordinate (q0) at (1,2);
		\coordinate (q1) at (2, 3);
		\coordinate (q2) at (3.5, 3.5);
		\coordinate (q3) at (5, 2.8);

		\coordinate (left) at (0, 0);
		\coordinate (q1') at (2, \reducedycoordinate);
		\coordinate (q2') at (3.5, \reducedycoordinate+0.5);
		\coordinate (q3') at (5, \reducedycoordinate-0.2);
		

		\coordinate (p0) at (1.5, 1.2);
		\coordinate (p1) at (3, 1);
		\coordinate (p2) at (4, 0.2);

		\draw[step=1cm,thick,dashed, color=red] (left) -- (q0) --  (q1);
		\draw[step=1cm,thick,dashed, color=blue] (q1) -- (q2) -- (q3);
		\draw[step=1cm,thick, color=red] (left) --(q1');
		\draw[step=1cm,thick, color=blue] (q1') -- (q2') -- (q3');
		\draw[step=1cm,thick] (p0) --  (p1) -- (p2);
		
		\draw [fill] (q0) circle [radius=0.05];		
		\draw [fill] (q1) circle [radius=0.05];		
		\draw [fill] (q2) circle [radius=0.05];		
		\draw [fill] (q3) circle [radius=0.05];		
		\draw [fill] (left) circle [radius=0.05];

		\draw [fill] (q1') circle [radius=0.05];			
		\draw [fill] (q2') circle [radius=0.05];		
		\draw [fill] (q3') circle [radius=0.05];		
		
		\draw [fill] (p0) circle [radius=0.05];		
		\draw [fill] (p1) circle [radius=0.05];		
		\draw [fill] (p2) circle [radius=0.05];		

		\draw[->, line width=0.6pt, color=red] (1, 1.5) -- (1,1);
		\draw[->, line width=0.6pt, color=blue] (3.5, 2.8) -- (3.5,2.3);
		


		
		\path (q3') ++(0.2, -0.4) node {$\HN(\maxslopered{\Vcal})$};
		\path (p2) ++(-0.2, -0.4) node {$\HN(\Wcal)$};
		\path (left) ++(-0.3, -0.05) node {$O$};

\end{tikzpicture}
\caption{Illustration of the maximal slope reduction}
\end{figure}

We note some basic properties of the maximal slope reduction. 

\begin{lemma}\label{basic properties of max slope reduction}
Let $\Vcal$ and $\Wcal$ be nonzero vector bundles on $\schff$ with integer slopes such that $\Vcal$ slopewise dominates $\Wcal$. Let $\maxslopered{\Vcal}$ denote the maximal slope reduction of $\Vcal$ to $\Wcal$. Then we have the following facts:
\begin{enumerate}[label=(\arabic*)]
\item $\mumax(\maxslopered{\Vcal}) = \mumax(\Wcal)$. 
\smallskip

\item $\rk(\maxslopered{\Vcal}) = \rk(\Vcal)$. 
\smallskip

\item $\Vcal = \maxslopered{\Vcal}$ if and only if $\mumax(\Vcal) = \mumax(\Wcal)$. 
\smallskip

\item $\maxslopered{\Vcal}$ slopewise dominates $\Wcal$. 
\smallskip

\item all slopes of $\maxslopered{\Vcal}$ are integers. 
\end{enumerate}
\end{lemma}
\begin{proof}
All statements follow immediately from Definition \ref{max slope reduction definition}.
\end{proof}

We also note a computational lemma that we will use.
\begin{lemma}[{\cite[Lemma 2.3.4]{Arizona_extvb}}]\label{cross product representation of degrees}
Let $\Vcal$ and $\Wcal$ be any vector bundles on $\schff$ with HN decompositions
\[\Vcal \simeq \bigoplus_{i=1}^p \trivbundle(\lambda_i)^{\oplus m_i} \quad\quad \text{ and } \quad\quad \Wcal \simeq \bigoplus_{j=1}^q \trivbundle(\kappa_j)^{\oplus n_j}\]
where $\lambda_1>\lambda_2 > \cdots > \lambda_p$ and $\kappa_1 > \kappa_2 > \cdots > \kappa_q$. We represent the $i$-th segment in $\HN(\Vcal)$ and $j$-th segment in $\HN(\Wcal)$ (from left to right) by the vectors $v_i$ and $w_j$, respectively. More precisely, we set
\[ v_i := \big(\rk(\trivbundle(\lambda_i)^{\oplus m_i}), \deg(\trivbundle(\lambda_i)^{\oplus m_i})\big) \quad\quad \text{ and } \quad\quad w_j := \big(\rk(\trivbundle(\kappa_j)^{\oplus n_j}), \deg(\trivbundle(\kappa_j)^{\oplus n_j})\big).\]

\begin{center}
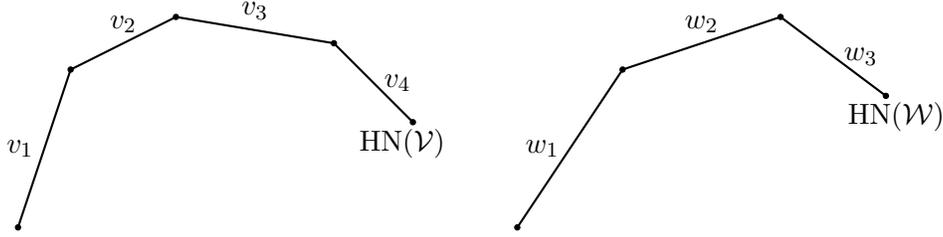
\begin{figure}[h]
\begin{tikzpicture}[scale=0.7]
		\coordinate (left) at (0, 0);
		\coordinate (q1) at (2.5, 3.5);
		\coordinate (q2) at (4, 4);
		\coordinate (q3) at (5, 3.7);

		\coordinate (p0) at (1, 3);
		\coordinate (p1) at (3, 4);
		\coordinate (p2) at (6, 3.5);
		\coordinate (p3) at (7.5, 2);
				

		\draw[step=1cm,thick] (left) -- node[left] {$v_1$} (p0);
		\draw[step=1cm,thick] (p0) -- node[above] {$v_2$} (p1);
		\draw[step=1cm,thick] (p1) -- node[above] {$v_3$} (p2);
		\draw[step=1cm,thick] (p2) -- node[right] {$v_4$} (p3);

		\draw [fill] (left) circle [radius=0.05];
		
		\draw [fill] (p0) circle [radius=0.05];		
		\draw [fill] (p1) circle [radius=0.05];		
		\draw [fill] (p2) circle [radius=0.05];		
		\draw [fill] (p3) circle [radius=0.05];		


		
		\path (p3) ++(-0.2, -0.4) node {$\HN(\Vcal)$};


\end{tikzpicture}
\hspace{.2in}
\begin{tikzpicture}[scale=0.7]
		\coordinate (left) at (0, 0);
		\coordinate (q1) at (2, 3);
		\coordinate (q2) at (5, 4);
		\coordinate (q3) at (7, 2.5);

		\coordinate (p0) at (0.5, 1.5);
		\coordinate (p1) at (1.5, 2.5);
		\coordinate (p2) at (3.5, 3);
		\coordinate (p3) at (4.5, 2.5);
		\coordinate (p4) at (5, 1.5);
				

		\draw[step=1cm,thick] (left) -- node[left] {$w_1$} (q1);
		\draw[step=1cm,thick] (q1) -- node[above] {$w_2$} (q2);
		\draw[step=1cm,thick] (q2) -- node[right] {$w_3$} (q3);

		\draw [fill] (q1) circle [radius=0.05];		
		\draw [fill] (q2) circle [radius=0.05];		
		\draw [fill] (q3) circle [radius=0.05];		
		\draw [fill] (left) circle [radius=0.05];
		


		
		\path (q3) ++(0.2, -0.4) node {$\HN(\Wcal)$};


\end{tikzpicture}
\setlength{\belowcaptionskip}{-0.3in}
\caption{Vector representation of HN polygons.}\label{diamond_Lemma_fig}
\end{figure}
\end{center}
If we write $\mu(v_i)$ and $\mu(w_j)$ respectively for the slopes of $v_i$ and $w_j$, we have an identity
\[ \deg(\Vcal^\vee \otimes \Wcal)^\nonneg = \sum_{\mu(v_i) \leq \mu(w_j)} v_i \times w_j\]
where $v_i \times w_j$ denotes the two-dimensional cross product of the vectors $v_i$ and $w_j$. In particular, we have $\deg(\Vcal^\vee \otimes \Wcal)^\nonneg = 0$ if $\mumin(\Vcal) \geq \mumax(\Wcal)$. 
\end{lemma}


Let us now proceed to the inductive part of our construction. 

\begin{prop}\label{degenerating sequence inductive part}
We can construct a sequence of vector bundles $\Ecal_1, \Ecal_2, \Ecal_3, \cdots$ so that the following statements hold for each $i = 1, 2, \cdots$. 

\begin{enumerate}[label = (\arabic*)]
\item\label{max common factor decomp for duals of Q and Ei} There exist decompositions 
\begin{equation*}
 \Qcal^\vee \simeq \Mcal_i \oplus \Rcal_i \quad\quad \text{ and } \quad\quad \Ecal_i^\vee \simeq \Mcal_i \oplus \Scal_i
\end{equation*}
which satisfy the following properties:

\begin{enumerate}[label=(\alph*)]
\item\label{slopewise dominance for complement part dual} $\Scal_i$ slopewise dominates $\Rcal_i$. 
\smallskip

\item\label{inequality for max slopes of complement parts dual} If $\Scal_i \neq 0$, we have $\mumax(\Scal_i)>\mumax(\Rcal_i)$. 
\smallskip

\item\label{ineqaulities for min slope of common factor dual} If $\Mcal_i \neq 0$ and $\Scal_i \neq 0$, we have $\mumin(\Mcal_i) \geq \mumax(\Scal_i) > \mumax(\Rcal_i)$.
\end{enumerate}
\smallskip

\item\label{construction of Ei} If $i>1$ we have
\[ \Ecal_i \simeq \begin{cases} \Qcal & \text{ if } \Ecal_{i-1} \simeq \Qcal\\ \Mcal_{i-1}^\vee \oplus \maxslopered{\Scal}_{i-1}^\vee & \text{ otherwise}\end{cases}\]
where $\maxslopered{\Scal}_{i-1}$ denotes the maximal slope reduction of $\Scal_{i-1}$ to $\Rcal_{i-1}$. 
\smallskip

\item\label{rank of Ei} $\rk(\Qcal) = \rk(\Ecal_i)$.
\smallskip

\item\label{integer slopes of Ei} all slopes of $\Ecal_i$ are integers.
\smallskip

\item\label{slopewise dominance of Q on Ei} $\Qcal$ slopewise dominates $\Ecal_i$. 


\end{enumerate}
\end{prop}

\begin{proof}
Let us take $\Ecal_1$ as in Proposition \ref{construction of E1}. We deduce the statements \ref{rank of Ei}, \ref{integer slopes of Ei} and \ref{slopewise dominance of Q on Ei} for $i=1$ directly follows from Proposition \ref{construction of E1}. Moreover, we obtain the statement \ref{max common factor decomp for duals of Q and Ei} for $i=1$ as a formal consequence of the statements \ref{rank of Ei} and \ref{slopewise dominance of Q on Ei} by Lemma \ref{implications of slopewise dominance}. Since the statement \ref{construction of Ei} for $i = 1$ vacuously holds, we have thus verified all statements for $i = 1$. 

We now proceed by induction on $i$. If $\Ecal_i \simeq \Qcal$, the induction step becomes trivial; in fact, if we take $\Ecal_{i+1} := \Qcal$, the statement \ref{construction of Ei} for $i+1$ is obvious by construction while the other statements \ref{max common factor decomp for duals of Q and Ei}, \ref{rank of Ei}, \ref{integer slopes of Ei} and \ref{slopewise dominance of Q on Ei} for $i+1$ immediately follow from the induction hypothesis as $\Ecal_{i+1} \simeq \Ecal_i$. We thus assume from now on that $\Ecal_i \not\simeq \Qcal$. By the induction hypothesis, the statement \ref{max common factor decomp for duals of Q and Ei} for $i$ yields decompositions 
\begin{equation}\label{max common slope decompositions for induction step}
\Qcal^\vee \simeq \Mcal_i \oplus \Rcal_i \quad\quad \text{ and } \quad\quad \Ecal_i^\vee \simeq \Mcal_i \oplus \Scal_i
\end{equation}
where $\Rcal_i$ slopewise dominates $\Scal_i$. Moreover, the statement \ref{integer slopes of Ei} for $i$ and the condition \ref{integer slopes of E, F, Q, reduced} in Proposition \ref{reduced key inequality} together imply that all slopes of $\Rcal_i$ and $\Scal_i$ are integers. Hence it makes sense to consider the maximal slope reduction of $\Scal_i$ to $\Rcal_i$, which we denote by $\maxslopered{\Scal}_i$. Let us now take
\begin{equation}\label{construction of Ei induction step}
\Ecal_{i+1}:= \Mcal_i^\vee \oplus \maxslopered{\Scal}_i^\vee.
\end{equation}
\begin{center}
\begin{figure}[h]
\begin{tikzpicture}	
		\coordinate (left) at (0, 0);
		\coordinate (q1) at (2.5, 3.5);
		\coordinate (q2) at (4, 4);
		\coordinate (q3) at (5, 3.7);

		\coordinate (p0) at (0.5, 1.5);
		\coordinate (p1) at (1.5, 2.5);
		\coordinate (p2) at (3.5, 3);
		\coordinate (p3) at (4.5, 2.5);
		\coordinate (p4) at (5, 1.5);
				
		\draw[step=1cm,thick, color=red] (left) -- (p0) --  (p1);
		\draw[step=1cm,thick, color=blue] (p1) -- (q1) -- (q2) -- (q3);
		\draw[step=1cm,thick, color=green] (p1) -- (p2) -- (p3) -- (p4);

		\draw [fill] (q1) circle [radius=0.05];		
		\draw [fill] (q2) circle [radius=0.05];		
		\draw [fill] (q3) circle [radius=0.05];		
		\draw [fill] (left) circle [radius=0.05];
		
		\draw [fill] (p0) circle [radius=0.05];		
		\draw [fill] (p1) circle [radius=0.05];		
		\draw [fill] (p2) circle [radius=0.05];		
		\draw [fill] (p3) circle [radius=0.05];		
		\draw [fill] (p4) circle [radius=0.05];		


		
		\path (q3) ++(0.2, -0.4) node {$\HN(\Ecal_i^\vee)$};
		\path (p4) ++(-0.2, -0.4) node {$\HN(\Qcal^\vee)$};

		\path (p0) ++(-0.2, 0.3) node {\color{red}$\Mcal_i$};
		\path (q1) ++(0, 0.3) node {\color{blue}$\Scal_i$};
		\path (p2) ++(0, -0.5) node {\color{green}$\Rcal_i$};

\end{tikzpicture}
\begin{tikzpicture}[scale=0.4]
        \pgfmathsetmacro{\textycoordinate}{6}
		\draw[->, line width=0.6pt] (0, \textycoordinate) -- (1.5,\textycoordinate);
		\draw (0,0) circle [radius=0.00];	
        \hspace{0.2cm}
\end{tikzpicture}
\begin{tikzpicture}
		\pgfmathsetmacro{\reducedycoordinate}{2.5/4+2.5}
	
		\coordinate (left) at (0, 0);
		\coordinate (q1) at (2.5, 3.5);
		\coordinate (q2) at (4, 4);
		\coordinate (q3) at (5, 3.7);

		\coordinate (q2') at (4, \reducedycoordinate);
		\coordinate (q3') at (5, \reducedycoordinate-0.3);

		\coordinate (p0) at (0.5, 1.5);
		\coordinate (p1) at (1.5, 2.5);
		\coordinate (p2) at (3.5, 3);
		\coordinate (p3) at (4.5, 2.5);
		\coordinate (p4) at (5, 1.5);
				
		\draw[step=1cm,thick, color=red] (left) -- (p0) --  (p1);
		\draw[step=1cm,thick,dashed, color=blue] (p1) -- (q1) -- (q2) -- (q3);
		\draw[step=1cm,thick, color=blue] (p1) -- (q2') -- (q3');
		\draw[step=1cm,thick, color=green] (p2) -- (p3) -- (p4);

		\draw [fill] (q1) circle [radius=0.05];		
		\draw [fill] (q2) circle [radius=0.05];		
		\draw [fill] (q3) circle [radius=0.05];		
		\draw [fill] (left) circle [radius=0.05];

		\draw [fill] (q2') circle [radius=0.05];		
		\draw [fill] (q3') circle [radius=0.05];	
		
		\draw [fill] (p0) circle [radius=0.05];		
		\draw [fill] (p1) circle [radius=0.05];		
		\draw [fill] (p2) circle [radius=0.05];		
		\draw [fill] (p3) circle [radius=0.05];		
		\draw [fill] (p4) circle [radius=0.05];	

		\draw[->, line width=0.6pt, color=blue] (3, 3.5) -- (3,3);	


		
		\path (q3') ++(1.1, 0) node {$\HN(\Ecal_{i+1}^\vee)$};
		\path (p4) ++(-0.2, -0.4) node {$\HN(\Qcal^\vee)$};

		\path (p0) ++(-0.2, 0.3) node {\color{red}$\Mcal_i$};
		\path (q2') ++(0, 0.3) node {\color{blue}$\maxslopered{\Scal}_i$};
		\path (p3) ++(-0.3, -0.5) node {\color{green}$\Rcal_i$};

\end{tikzpicture}
\setlength{\belowcaptionskip}{-0.3in}
\caption{Construction of the sequence $(\Ecal_i)$}\label{construction of the sequence Ei}
\end{figure}
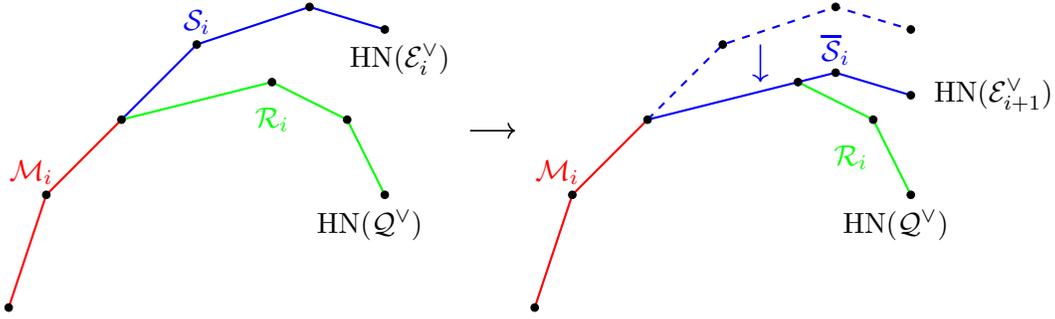
\end{center}
The statement \ref{construction of Ei} for $i+1$ is obvious by our definition of $\Ecal_{i+1}$ in \eqref{construction of Ei induction step}. We also verify the statement \ref{rank of Ei} for $i+1$ by computing
\begin{equation}\label{rank computation for Ei+1}
\rk(\Ecal_{i+1}) = \rk(\Mcal_i) + \rk(\maxslopered{\Scal}_i) = \rk(\Mcal_i) + \rk(\Scal_i) = \rank(\Ecal_i) = \rk(\Qcal)
\end{equation}
where for each equality we use \eqref{construction of Ei induction step}, Lemma \ref{basic properties of max slope reduction}, \eqref{max common slope decompositions for induction step} and the statement \ref{rank of Ei} for $i$. Moreover, since all slopes of $\Mcal_i$ and $\Scal_i$ are integers by the statement \ref{integer slopes of Ei} for $i$, we obtain the statement \ref{integer slopes of Ei} for $i+1$ using \eqref{construction of Ei induction step} and Lemma \ref{basic properties of max slope reduction}. Furthermore, since $\maxslopered{\Scal}_i$ slopewise dominates $\Rcal_i$ by Lemma \ref{basic properties of max slope reduction}, we deduce slopewise dominance of $\Ecal_{i+1}^\vee$ on $\Qcal^\vee$ from decompositions 
\[\Ecal_{i+1}^\vee = \Mcal_i \oplus \maxslopered{\Scal}_i \quad\quad \text{ and } \quad\quad \Qcal^\vee = \Mcal_i \oplus \Rcal_i\]
as given by \eqref{max common slope decompositions for induction step} and \eqref{construction of Ei induction step}, and consequently verify the statement \ref{max common factor decomp for duals of Q and Ei} for $i+1$ by Lemma \ref{implications of slopewise dominance}. 
In addition, slopewise dominance of $\Ecal_{i+1}^\vee$ on $\Qcal^\vee$ implies the statement \ref{slopewise dominance of Q on Ei} for $i+1$ by Lemma \ref{implications of slopewise dominance} and \eqref{rank computation for Ei+1}. We thus have all statements in Proposition \ref{degenerating sequence inductive part} for $i+1$, thereby concluding our proof by induction. 
\end{proof}

\begin{remark}
From our construction it is not hard to see that $\Ecal_{i+1}^\vee$ slopewise dominates $\Ecal_i^\vee$ for all $i$, as we proposed while sketching our proof of Proposition \ref{reduced key inequality}. Although this ``dually degenerating" property won't explicitly appear in our argument, it will play a crucial role in the proof of Proposition \ref{decreasing cEF(Q) for degenerating sequence} under the guise of relations between slopes of $\Scal_i$ and $\maxslopered{\Scal}_i$. 
\end{remark}

We record a couple of simple but useful observations about the construction in Proposition \ref{degenerating sequence inductive part}.
\begin{lemma}\label{simple observations about Ei}
Let $\Ecal_1, \Ecal_2, \cdots$ be as in Proposition \ref{degenerating sequence inductive part}. For each $i = 1, 2, \cdots$ we take decompositions
\begin{equation*}
\Qcal^\vee \simeq \Mcal_i \oplus \Rcal_i \quad\quad \text{ and } \quad\quad \Ecal_i^\vee \simeq \Mcal_i \oplus \Scal_i
\end{equation*}
as given by the statement \ref{max common factor decomp for duals of Q and Ei} in Proposition \ref{degenerating sequence inductive part}. 
\begin{enumerate}[label=(\arabic*)]
\item The bundle $\Mcal_i$ represents the common part of $\HN(\Qcal^\vee)$ and $\HN(\Ecal_i^\vee)$.
\smallskip

\item If $\Ecal_i \not\simeq \Qcal$, we have $\Rcal_i \neq 0$ and $\Scal_i \neq 0$. 
\end{enumerate}
\end{lemma}

\begin{proof}
The first statement follows from the remark after Lemma \ref{implications of slopewise dominance}. The second statement is an immediate consequence of the first statement. 
\end{proof}

Our construction process turns out to be essentially finite in the sense that the sequence stabilizes after finitely many steps, as the following proposition shows. 

\begin{prop}\label{degenerating sequence finiteness}
Let $\Ecal_1, \Ecal_2, \cdots$ be as in Proposition \ref{degenerating sequence inductive part}. There exists $r >0$ such that $\Ecal_i \simeq \Qcal$ for all $i \geq r$. 
\end{prop}

\begin{proof}
By the statement \ref{construction of Ei} in Proposition \ref{degenerating sequence inductive part}, it suffices to show that $\Ecal_i \simeq \Qcal$ for some $i$. Suppose for contradiction that $\Ecal_i \not\simeq \Qcal$ for all $i$. Let us take decompositions 
\begin{equation}\label{max common slope decompositions for stabilization}
\Qcal^\vee \simeq \Mcal_i \oplus \Rcal_i \quad\quad \text{ and } \quad\quad \Ecal_i^\vee \simeq \Mcal_i \oplus \Scal_i
\end{equation}
as given by the statement \ref{max common factor decomp for duals of Q and Ei} in Proposition \ref{degenerating sequence inductive part}. As $\Ecal_i \not\simeq \Qcal$ by our assumption, the statement \ref{construction of Ei} in Proposition \ref{degenerating sequence inductive part} yields a decomposition
\begin{equation}\label{construction of Ei for stabiliztion}
\Ecal_{i+1}^\vee \simeq \Mcal_i \oplus \maxslopered{\Scal}_i
\end{equation}
where $\maxslopered{\Scal}_i$ denotes the maximal slope reduction of $\Scal_i$ to $\Rcal_i$. Since $\Rcal_i$ and $\Scal_i$ are both nonzero by Lemma \ref{simple observations about Ei}, the polygons $\HN(\Rcal_i)$ and $\HN(\maxslopered{\Scal}_i)$ must have a nontrivial common part, which we represent by a nonzero vector bundle $\Tcal_i$ on $\schff$. Then we deduce from the decompositions \eqref{max common slope decompositions for stabilization} and \eqref{construction of Ei for stabiliztion}
that the common part of $\HN(\Qcal^\vee)$ and $\HN(\Ecal_{i+1}^\vee)$ should include $\HN(\Mcal_i \oplus \Tcal_i)$. 
Now by Lemma \ref{simple observations about Ei} we find
\[\rk(\Mcal_{i+1}) \geq \rk(\Mcal_i \oplus \Tcal_i) = \rk(\Mcal_i) + \rk(\Tcal_i) > \rk(\Mcal_i).\]
In particular, the sequence $\big(\rk(\Mcal_i)\big)$ should be unbounded. However, this is impossible since we have $\rk(\Mcal_i) \leq \rk(\Qcal)$ by \eqref{max common slope decompositions for stabilization}. We thus complete the proof by contradiction. 
\end{proof}

We now prove the essential property of our sequence. 

\begin{prop}\label{decreasing cEF(Q) for degenerating sequence}
Let $\Ecal_1, \Ecal_2, \cdots$ be as in Proposition \ref{degenerating sequence inductive part}. For each $i = 1, 2, \cdots,$ we take decompositions
\begin{equation}\label{max common slope decompositions for decreasing cEF(Q)}
\Qcal^\vee \simeq \Mcal_i \oplus \Rcal_i \quad\quad \text{ and } \quad\quad \Ecal_i^\vee \simeq \Mcal_i \oplus \Scal_i
\end{equation}
as given by the statement \ref{max common factor decomp for duals of Q and Ei} in Proposition \ref{degenerating sequence inductive part}. Then for each $i = 1, 2, \cdots$, we have an inequality
\begin{equation}\label{inductive inequality for cEF(Q)}
c_{\Ecal_i, \Fcal}(\Qcal) \geq c_{\Ecal_{i+1}, \Fcal}(\Qcal)
\end{equation}
where equality holds only if $\Ecal_i \simeq \Qcal$ or $\rk(\Scal_i^\vee) = \rk(\Fcal^{>\mumin(\Scal_i^\vee)})$. 
\end{prop}

\begin{proof}
If $\Ecal_i \simeq \Qcal$, the proof is trivial as we have $\Ecal_{i+1} \simeq \Qcal \simeq \Ecal_i$ by the statement \ref{construction of Ei} in Proposition \ref{degenerating sequence inductive part}. We thus assume from now on that $\Ecal_i \not\simeq \Qcal$. 

Let $\maxslopered{\Scal}_i$ denote the maximal slope reduction of $\Scal_i$ to $\Rcal_i$. Then we have a decomposition
\begin{equation}\label{construction of Ei for decreasing cEF(Q)}
\Ecal_{i+1}^\vee \simeq \Mcal_i \oplus \maxslopered{\Scal}_i
\end{equation}
by the statement \ref{construction of Ei} in Proposition \ref{degenerating sequence inductive part}. Moreover, since $\Ecal_i \not\simeq \Qcal$ we have $\Rcal_i \neq 0$ and $\Scal_i \neq 0$ by Lemma \ref{simple observations about Ei}. Hence the statement \ref{max common factor decomp for duals of Q and Ei} in Proposition \ref{degenerating sequence inductive part} yields
\begin{equation}\label{order of slopes in max common factor decomps}
\mumin(\Mcal_i) \geq \mumax(\Scal_i) > \mumax(\Rcal_i) = \mumax(\maxslopered{\Scal}_i) \quad\quad \text{ if } \Mcal_i \neq 0.
\end{equation}
Then by Lemma \ref{cross product representation of degrees} we obtain
\begin{equation}\label{zero degree formula for max common factor decomps}
\deg(\Mcal_i^\vee \otimes \Scal_i)^\nonneg = \deg(\Mcal_i^\vee \otimes \maxslopered{\Scal}_i)^\nonneg = 0.
\end{equation}
Now we use \eqref{max common slope decompositions for decreasing cEF(Q)}, \eqref{construction of Ei for decreasing cEF(Q)} and \eqref{zero degree formula for max common factor decomps} to find
\begin{equation*}
\begin{aligned}
\deg(\Ecal_i^\vee \otimes \Fcal)^\nonneg &= \deg((\Mcal_i \oplus \Scal_i) \otimes \Fcal)^\nonneg
= \deg(\Mcal_i \otimes \Fcal)^\nonneg + \deg(\Scal_i \oplus \Fcal)^\nonneg,\\
\deg(\Ecal_{i+1}^\vee \otimes \Fcal)^\nonneg &= \deg((\Mcal_i \oplus \maxslopered{\Scal}_i) \otimes \Fcal)^\nonneg
= \deg(\Mcal_i \otimes \Fcal)^\nonneg + \deg(\maxslopered{\Scal}_i \otimes \Fcal)^\nonneg,\\
\deg(\Ecal_i^\vee \otimes \Qcal)^\nonneg &= \deg((\Mcal_i \oplus \Scal_i) \otimes (\Mcal_i^\vee \oplus \Rcal_i^\vee))^\nonneg\\
&=  \deg(\Mcal_i \otimes \Mcal_i^\vee)^\nonneg + \deg(\Mcal_i \otimes \Rcal_i^\vee)^\nonneg + \deg(\Scal_i \otimes \Mcal_i^\vee)^\nonneg + \deg(\Scal_i \otimes \Rcal_i^\vee)^\nonneg \\
&=  \deg(\Mcal_i \otimes \Mcal_i^\vee)^\nonneg + \deg(\Mcal_i \otimes \Rcal_i^\vee)^\nonneg + \deg(\Scal_i \otimes \Rcal_i^\vee)^\nonneg, \\
\deg(\Ecal_{i+1}^\vee \otimes \Qcal)^\nonneg &= \deg((\Mcal_i \oplus \maxslopered{\Scal}_i) \otimes (\Mcal_i^\vee \oplus \Rcal_i^\vee))^\nonneg\\
&=  \deg(\Mcal_i \otimes \Mcal_i^\vee)^\nonneg + \deg(\Mcal_i \otimes \Rcal_i^\vee)^\nonneg + \deg(\maxslopered{\Scal}_i \otimes \Mcal_i^\vee)^\nonneg + \deg(\maxslopered{\Scal}_i \otimes \Rcal_i^\vee)^\nonneg \\
&=  \deg(\Mcal_i \otimes \Mcal_i^\vee)^\nonneg + \deg(\Mcal_i \otimes \Rcal_i^\vee)^\nonneg + \deg(\maxslopered{\Scal}_i \otimes \Rcal_i^\vee)^\nonneg.
\end{aligned}
\end{equation*}
Therefore by Definition \ref{definition of cEF(Q)} we have
\begin{align*}
c_{\Ecal_i, \Fcal}(\Qcal) - c_{\Ecal_{i+1}, \Fcal}(\Qcal) &= \big(\deg(\Ecal_i^\vee \otimes \Fcal)^\nonneg - \deg(\Ecal_{i+1}^\vee \otimes \Fcal)^\nonneg\big) - \big(\deg(\Ecal_i^\vee \otimes \Qcal)^\nonneg - \deg(\Ecal_{i+1}^\vee \otimes \Qcal)^\nonneg\big)\\
&= \big(\deg(\Scal_i \otimes \Fcal)^\nonneg - \deg(\maxslopered{\Scal}_i \otimes \Fcal)^\nonneg\big) - \big(\deg(\Scal_i \otimes \Rcal_i^\vee)^\nonneg - \deg(\maxslopered{\Scal}_i \otimes \Rcal_i^\vee)^\nonneg \big)
\end{align*}
The desired inequality \eqref{inductive inequality for cEF(Q)} is thus equivalent to 
\begin{equation}\label{decreasing cEF(Q) first reformulation}
\deg(\Scal_i \otimes \Fcal)^\nonneg - \deg(\maxslopered{\Scal}_i \otimes \Fcal)^\nonneg \geq \deg(\Scal_i \otimes \Rcal_i^\vee)^\nonneg - \deg(\maxslopered{\Scal}_i \otimes \Rcal_i^\vee)^\nonneg. 
\end{equation}

Let us set $\lambda := \mumax(\Rcal_i)$ and $r := \rk(\Scal_i^{>\lambda})$. 
Since $\maxslopered{\Scal}_i$ is the maximal slope reduction of $\Scal_i$, we have decompositions 
\begin{equation*}\label{max slope decomps for max slope reduction} 
\Scal_i \simeq \Scal_i^{> \lambda} \oplus \Scal_i^{\leq \lambda} \quad\quad \text{ and } \quad\quad \maxslopered{\Scal}_i \simeq \trivbundle(\lambda)^{\oplus r} \oplus \Scal_i^{\leq \lambda}.
\end{equation*}
Then we have
\begin{align*}
\deg(\Scal_i \otimes \Fcal)^\nonneg &= \deg((\Scal_i^{> \lambda} \oplus \Scal_i^{\leq \lambda}) \otimes \Fcal)^\nonneg\\
&= \deg(\Scal_i^{> \lambda} \otimes \Fcal)^\nonneg + \deg(\Scal_i^{\leq \lambda} \otimes \Fcal)^\nonneg,\\
\deg(\maxslopered{\Scal}_i \otimes \Fcal)^\nonneg &= \deg((\trivbundle(\lambda)^{\oplus r} \oplus \Scal_i^{\leq \lambda}) \otimes \Fcal)^\nonneg\\
&= \deg(\trivbundle(\lambda)^{\oplus r} \otimes \Fcal)^\nonneg + \deg(\Scal_i^{\leq \lambda} \otimes \Fcal)^\nonneg,\\
\deg(\Scal_i \otimes \Rcal_i^\vee)^\nonneg &= \deg((\Scal_i^{> \lambda} \oplus \Scal_i^{\leq \lambda}) \otimes \Rcal_i^\vee)^\nonneg\\
&= \deg(\Scal_i^{> \lambda} \otimes \Rcal_i^\vee)^\nonneg + \deg(\Scal_i^{\leq \lambda} \otimes \Rcal_i^\vee)^\nonneg,\\
\deg(\maxslopered{\Scal}_i \otimes \Rcal_i^\vee)^\nonneg &= \deg((\trivbundle(\lambda)^{\oplus r} \oplus \Scal_i^{\leq \lambda}) \otimes \Rcal_i^\vee)^\nonneg\\
&= \deg(\trivbundle(\lambda)^{\oplus r} \otimes \Rcal_i^\vee)^\nonneg + \deg(\Scal_i^{\leq \lambda} \otimes \Rcal_i^\vee)^\nonneg. 
\end{align*}
We can thus rewrite the inequality \eqref{decreasing cEF(Q) first reformulation} as
\begin{equation}\label{decreasing cEF(Q) second reformulation}
\deg(\Scal_i^{> \lambda} \otimes \Fcal)^\nonneg - \deg(\trivbundle(\lambda)^{\oplus r} \otimes \Fcal)^\nonneg  \geq \deg(\Scal_i^{> \lambda} \otimes \Rcal_i^\vee)^\nonneg - \deg(\trivbundle(\lambda)^{\oplus r} \otimes \Rcal_i^\vee)^\nonneg.
\end{equation}

Let us now take sequences of vectors $(r_a), (s_b)$ and $(f_c)$ which respectively represent the line segments in $\HN(\Rcal_i), \HN(\Scal_i^{> \lambda})$ and $\HN(\Fcal^\vee)$. Let us also set $s:= \sum s_b$ and take $\maxslopered{s}$ to be a vector representing the only line segment in $\HN(\trivbundle(\lambda)^{\oplus r})$. By construction, we obtain
\begin{align*}
s &= (\rk(\Scal_i^{> \lambda}), \deg(\Scal_i^{> \lambda})) = (\rk(\Scal_i^{> \lambda}), \mu(\Scal_i^{> \lambda})\cdot \rk(\Scal_i^{> \lambda})) = (r, r \mu(\Scal_i^{>\lambda})),\\
\maxslopered{s} &= (\rk(\trivbundle(\lambda)^{\oplus r})), \deg(\trivbundle(\lambda)^{\oplus r}) = (\rk(\trivbundle(\lambda)^{\oplus r}), \mu(\trivbundle(\lambda)^{\oplus r}) \cdot \rk(\trivbundle(\lambda)^{\oplus r})) = (r, r \lambda).
\end{align*}
where we use the fact that $\lambda = \mumax(\Rcal_i)$ is an integer by the decomposition \eqref{max common slope decompositions for decreasing cEF(Q)} and the condition \ref{integer slopes of E, F, Q, reduced} in Proposition \ref{reduced key inequality}. We thus have
\begin{equation}\label{difference of vectors representing max slope reduction} 
s - \maxslopered{s} = (0, r (\mu(\Scal_i^{> \lambda}) - \lambda)).
\end{equation}

We now aim to estimate the left side of \eqref{decreasing cEF(Q) second reformulation}. By Lemma \ref{cross product representation of degrees}, we may write
\begin{align}
\deg(\Scal_i^{> \lambda} \otimes \Fcal)^\nonneg - \deg(\trivbundle(\lambda)^{\oplus r} \otimes \Fcal)^\nonneg &= \deg(\Scal_i^{> \lambda} \otimes (\Fcal^\vee)^\vee)^\nonneg - \deg(\trivbundle(\lambda)^{\oplus r} \otimes (\Fcal^\vee)^\vee)^\nonneg \nonumber\\ 
&= \sum_{\mu(f_c) \leq \mu(s_a)} f_c \times s_a - \sum_{\mu(f_c) \leq \lambda} f_c \times \maxslopered{s}. \label{left side of reduced induction inequality}
\end{align}
Note that each $s_a$ satisfies $\mu(s_a) > \lambda$ by construction. Hence each $f_c$ with $\mu(f_c) \leq \lambda$ must satisfy $\mu(f_c) \leq \mu(s_a)$ for all $s_a$'s. We thereby obtain an inequality 
\begin{equation}\label{max slope reduction s, f terms inequality}
\sum_{\mu(f_c) \leq \lambda} f_c \times s_a \leq \sum_{\mu(f_c) \leq \mu(s_a)} f_c \times s_a
\end{equation}
as every term on each side is nonnegative. Now \eqref{left side of reduced induction inequality} yields
\begin{align}
\deg(\Scal_i^{> \lambda} \otimes \Fcal)^\nonneg - \deg(\trivbundle(\lambda)^{\oplus r} \otimes \Fcal)^\nonneg & \geq \sum_{\mu(f_c) \leq \lambda} f_c \times s_a - \sum_{\mu(f_c) \leq \lambda} f_c \times \maxslopered{s} \nonumber\\
&= \sum_{\mu(f_c) \leq \lambda} f_c \times \sum s_a - \sum_{\mu(f_c) \leq \lambda} f_c \times \maxslopered{s} \nonumber\\
&= \sum_{\mu(f_c) \leq \lambda} f_c \times (s - \maxslopered{s}). \label{left side of reduced induction inequality lower bound}
\end{align}
Moreover, as the sequence $(f_c)$ represents the line segments in $\HN(\Fcal^\vee)$ we have
\[ \sum_{\mu(f_c) \leq \lambda} f_c = (\rk((\Fcal^\vee)^{\leq \lambda}), \deg((\Fcal^\vee)^{\leq \lambda})),\]
and consequently obtain
\[ \sum_{\mu(f_c) \leq \lambda} f_c \times (s - \maxslopered{s}) = r \cdot \rk((\Fcal^\vee)^{\leq \lambda}) \cdot  (\mu(\Scal_i^{> \lambda}) - \lambda)\]
by \eqref{difference of vectors representing max slope reduction}. We can thus rewrite \eqref{left side of reduced induction inequality lower bound} as
\begin{equation}\label{reduced inequality left side simplified lower bound}
\deg(\Scal_i^{> \lambda} \otimes \Fcal)^\nonneg - \deg(\trivbundle(\lambda)^{\oplus r} \otimes \Fcal)^\nonneg \geq r \cdot \rk((\Fcal^\vee)^{\leq \lambda}) \cdot  (\mu(\Scal_i^{> \lambda}) - \lambda). 
\end{equation}

Our next task is to compute the right side of \eqref{decreasing cEF(Q) second reformulation}. By Lemma \ref{cross product representation of degrees} we have
\begin{equation}\label{right side of reduced induction inequality}
\deg(\Scal_i^{> \lambda} \otimes \Rcal_i^\vee)^\nonneg - \deg(\trivbundle(\lambda)^{\oplus r} \otimes \Rcal_i^\vee)^\nonneg = \sum_{\mu(r_b) \leq \mu(s_a)} r_b \times s_a - \sum_{\mu(r_b) \leq \lambda} r_b \times \maxslopered{s}. 
\end{equation}
Since the sequences $(s_a)$ and $(r_b)$ respectively represent the line segments in $\HN(\Scal_i^{> \lambda})$ and $\HN(\Rcal_i)$, we have
\[\mu(r_b) \leq \mumax(\Rcal_i) = \lambda < \mu(s_a)\]
for all $s_a$'s and $r_b$'s. Hence we can simplify \eqref{right side of reduced induction inequality} as
\begin{align}
\deg(\Scal_i^{> \lambda} \otimes \Rcal_i^\vee)^\nonneg - \deg(\trivbundle(\lambda)^{\oplus r} \otimes \Rcal_i^\vee)^\nonneg &= \sum r_b \times s_a - \sum r_b \times \maxslopered{s} \nonumber\\
&= \sum r_b \times \sum s_a - \sum r_b \times \maxslopered{s} \nonumber\\
&= \sum r_b \times (s - \maxslopered{s}). \label{right side of reduced induction inequality simplified}
\end{align}
Moreover, by construction we have
\[ \sum r_b = (\rk(\Rcal_i), \deg(\Rcal_i)),\]
and consequently obtain
\[ \sum r_b \times (s - \maxslopered{s}) = r \cdot \rk(\Rcal_i) \cdot  (\mu(\Scal_i^{> \lambda}) - \lambda)\]
by \eqref{difference of vectors representing max slope reduction}. We can thus rewrite \eqref{right side of reduced induction inequality simplified} as
\begin{equation}\label{right side of reduced induction inequality fully simplified}
\deg(\Scal_i^{> \lambda} \otimes \Rcal_i^\vee)^\nonneg - \deg(\trivbundle(\lambda)^{\oplus r} \otimes \Rcal_i^\vee)^\nonneg =  r \cdot \rk(\Rcal_i) \cdot  (\mu(\Scal_i^{> \lambda}) - \lambda). 
\end{equation}

Since $\lambda = \mumax(\Rcal_i)$, we use \eqref{max common slope decompositions for decreasing cEF(Q)} to find
\[ \rk(\Rcal_i) = \rk(\Rcal_i^{\leq \lambda}) \leq \rk((\Qcal^\vee)^{\leq \lambda}).\]
In addition, by Lemma \ref{rank and degree of dual bundle}, Proposition \ref{equivalence of two characterizations for subbundles} and the condition \ref{slopewise dominance of F on Q, reduced} in Proposition \ref{reduced key inequality} we find
\[\rk((\Qcal^\vee)^{\leq \lambda}) = \rk(\Qcal^{\geq -\lambda}) \leq \rk(\Fcal^{\geq -\lambda}) = \rk((\Fcal^\vee)^{\leq \lambda}).\]
Hence we have
\begin{equation*}
\rk(\Rcal_i) \leq \rk((\Fcal^\vee)^{\leq \lambda}).
\end{equation*}
Furtheremore, since $r = \rk(\Scal_i^{> \lambda}) >0$ by \eqref{order of slopes in max common factor decomps} and $\mu(\Scal_i^{>\lambda}) - \lambda >0$ by definition, we obtain
\begin{equation}\label{comparison of two estimates in reduced induction inequality}
r \cdot \rk((\Fcal^\vee)^{\leq \lambda}) \cdot  (\mu(\Scal_i^{> \lambda}) - \lambda) \geq r\cdot \rk(\Rcal_i) \cdot  (\mu(\Scal_i^{> \lambda}) - \lambda).
\end{equation}
Combining this with \eqref{reduced inequality left side simplified lower bound} and \eqref{right side of reduced induction inequality fully simplified}, we deduce the inequality \eqref{decreasing cEF(Q) second reformulation} which is equivalent to the desired inequality \eqref{decreasing cEF(Q) first reformulation}.

Let us now consider the equality condition. From \eqref{comparison of two estimates in reduced induction inequality}, we need 
\begin{equation}\label{induction inequality first equality condition} 
\rk(\Rcal_i) = \rk((\Fcal^\vee)^{\leq \lambda})
\end{equation}
since both $r$ and $\mu(\Scal_i^{\lambda}) - \lambda$ are positive as already noted. We also need equality in \eqref{reduced inequality left side simplified lower bound}, which requires equality in \eqref{max slope reduction s, f terms inequality}. Since every term on each side of \eqref{max slope reduction s, f terms inequality} is nonnegative, we must have identical nonzero terms on both sides of \eqref{max slope reduction s, f terms inequality}. In particular, every $f_c$ with $\mu(f_c) < \mumax(\Scal_i^{>\lambda})$ must satisfy $\mu(f_c) \leq \lambda$; indeed, for such an $f_c$ we have a nonzero term $f_c \times s_a$ on the right side for some $s_a$ with $\mu(s_a) = \mumax(\Scal_i^{>\lambda})$, and therefore must have the same nonzero term on the left side. We thus obtain
\begin{equation}\label{induction inequality second equality condition}
\rk((\Fcal^\vee)^{\leq \lambda}) =  \rk((\Fcal^\vee)^{<\mumax(\Scal_i^{> \lambda})}) = \rk((\Fcal^\vee)^{<\mumax(\Scal_i)}) 
\end{equation}
where for the second equality we observe $\mumax(\Scal_i^{> \lambda}) = \mumax(\Scal_i)$ by \eqref{order of slopes in max common factor decomps}. Moreover, by Lemma \ref{rank, degree and dual of stable bundles} and Lemma \ref{rank and degree of dual bundle} we have
\begin{equation}\label{rewriting of equality condition for induction inequality F part}
\rk((\Fcal^\vee)^{< \mumax(\Scal_i)}) = \rk((\Fcal^\vee)^{< - \mumin(\Scal_i^\vee)}) = \rk(\Fcal^{> \mumin(\Scal_i^\vee)}).
\end{equation}
We also have
\begin{equation}\label{rewriting of equality condition for induction inequality R part}
\rk(\Rcal_i) = \rk(\Scal_i) = \rk(\Scal_i^\vee)
\end{equation}
by \eqref{max common slope decompositions for decreasing cEF(Q)}, the statement \ref{rank of Ei} in Proposition \ref{degenerating sequence inductive part} and Lemma \ref{rank and degree of dual bundle}. We then combine \eqref{induction inequality first equality condition}, \eqref{induction inequality second equality condition}, \eqref{rewriting of equality condition for induction inequality F part} and \eqref{rewriting of equality condition for induction inequality R part} to obtain an equality condition
\[ \rk(\Scal_i^\vee) = \rk(\Fcal^{> \mumin(\Scal_i^\vee)})\]
as desired. 
\end{proof}

\begin{prop}\label{strict decreasing cEF(Q) in first two steps}
Let $\Ecal_1, \Ecal_2, \cdots$ be as in Proposition \ref{degenerating sequence inductive part}. Then we have a strict inequality
\[c_{\Ecal, \Fcal}(\Qcal) > c_{\Ecal_2, \Fcal}(\Qcal).\]
\end{prop}

\begin{proof}
Proposition \ref{construction of E1} and Proposition \ref{decreasing cEF(Q) for degenerating sequence} together yield
\begin{equation}\label{cEF(Q) inequality for first two steps} 
c_{\Ecal, \Fcal}(\Qcal) \geq c_{\Ecal_1, \Fcal}(\Qcal) \geq c_{\Ecal_2, \Fcal}(\Qcal).
\end{equation}
We need to prove that at least one of the inequalities in \eqref{cEF(Q) inequality for first two steps} must be strict. We assume for contradiction that
\begin{equation}\label{hypothetical equalities for first two steps} 
c_{\Ecal, \Fcal}(\Qcal) = c_{\Ecal_1, \Fcal}(\Qcal) = c_{\Ecal_2, \Fcal}(\Qcal).
\end{equation}

Let us first consider the case $\Ecal_1 \simeq \Qcal$. We note that
\begin{equation}\label{first step degree of Q}
\deg(\Qcal)^\nonneg = \deg(\Ecal_1)^\nonneg = 0
\end{equation}
where the second equality follows from the statement \ref{max slope of E1} in Proposition \ref{construction of E1}. In addition, the first equality in \eqref{hypothetical equalities for first two steps} yields
\begin{equation}\label{first step equality condition}
\deg(\Fcal)^\nonneg = \deg(\Qcal)^\nonneg.
\end{equation}
by the statement \ref{decreasing c for E1} in Proposition \ref{construction of E1}. Now \eqref{first step degree of Q} and \eqref{first step equality condition} together yield $\deg(\Fcal)^\nonneg = 0$, which in particular implies $\mumax(\Fcal) \leq 0$. However, this is impossible because of the conditions \ref{slopewise dominance of F on E, reduced}, \ref{no common slopes for E and F, reduced} and \ref{zero max slope for E, reduced} in Proposition \ref{reduced key inequality}. We have thus obtained a desired contradiction. 

Now it remains to consider the case $\Ecal_1 \not\simeq \Qcal$. Let us take decompositions 
\[ \Qcal^\vee \simeq \Mcal_1 \oplus \Rcal_1 \quad\quad \text{ and } \quad\quad \Ecal_1^\vee \simeq \Mcal_1 \oplus \Scal_1\]
as given by the statement \ref{max common factor decomp for duals of Q and Ei} in Proposition \ref{degenerating sequence inductive part}. As $\Ecal = \Ecal_1 \oplus \trivbundle$ by Proposition \ref{construction of E1}, we obtain
\begin{equation}\label{decomp of E first step}
\Ecal \simeq \Mcal_1^\vee \oplus \Scal_1^\vee \oplus \trivbundle.
\end{equation}
Since $\Scal_1^\vee$ is a direct summand of $\Ecal$, we have
\[\mumin(\Scal_1^\vee) \leq \mumax(\Ecal) = 0\]
by the condition \ref{zero max slope for E, reduced} in Proposition \ref{reduced key inequality}. Hence \eqref{decomp of E first step} yields
\begin{equation}\label{rank bound of E first step}
\rk(\Ecal^{\geq \mumin(\Scal_1^\vee)}) \geq \rk(\Scal_1^\vee \oplus \trivbundle) > \rk(\Scal_1^\vee)
\end{equation}
Moreover, as $\Ecal_1 \not\simeq \Qcal$, Proposition \ref{decreasing cEF(Q) for degenerating sequence} and the second equality in \eqref{hypothetical equalities for first two steps} together imply
\begin{equation}\label{second step equality condition}
\rk(\Scal_1^\vee) = \rk(\Fcal^{> \mumin(\Scal_1^\vee)}).
\end{equation}
Since $\mumin(\Scal_1^\vee)$ is a slope of $\Ecal$ by \eqref{decomp of E first step}, it is not a slope of $\Fcal$ by the condition \ref{no common slopes for E and F, reduced} in Proposition \ref{reduced key inequality}. Hence we have
\begin{equation}\label{F does not have the min slope of S1}
\rk(\Fcal^{> \mumin(\Scal_1^\vee)}) = \rk(\Fcal^{\geq \mumin(\Scal_1^\vee)}).
\end{equation}
Now we combine \eqref{rank bound of E first step}, \eqref{second step equality condition} and \eqref{F does not have the min slope of S1} to obtain
\[ \rk(\Ecal^{\geq \mumin(\Scal_1^\vee)}) > \rk(\Fcal^{\geq \mumin(\Scal_1^\vee)}).\]
However, this is impossible because of the condition \ref{slopewise dominance of F on E} in Proposition \ref{reduced key inequality} and Proposition \ref{equivalence of two characterizations for subbundles}. We thus complete the proof by contradiction. 
\end{proof}

Since $c_{\Qcal, \Fcal}(\Qcal) = 0$ by Definition \ref{definition of cEF(Q)}, we deduce Proposition \ref{reduced key inequality} from Proposition \ref{degenerating sequence finiteness}, Proposition \ref{decreasing cEF(Q) for degenerating sequence} and Proposition \ref{strict decreasing cEF(Q) in first two steps}. This concludes our proof of Theorem \ref{classification of subbundles}.

\bibliographystyle{amsalpha}

\bibliography{Bibliography}

\newcommand{\etalchar}[1]{$^{#1}$}
\providecommand{\bysame}{\leavevmode\hbox to3em{\hrulefill}\thinspace}
\providecommand{\MR}{\relax\ifhmode\unskip\space\fi MR }
\providecommand{\MRhref}[2]{%
  \href{http://www.ams.org/mathscinet-getitem?mr=#1}{#2}
}
\providecommand{\href}[2]{#2}
\begin{thebibliography}{BFH{\etalchar{+}}17}

\bibitem[BFH{\etalchar{+}}17]{Arizona_extvb}
Christopher Birkbeck, Tony Feng, David Hansen, Serin Hong, Qirui Li, Anthony
  Wang, and Lynelle Ye, \emph{Extensions of vector bundles on the
  {F}argues-{F}ontaine curve}, J. Inst. Math. Jussieu, to appear.

\bibitem[Col02]{Colmez_BCspace}
Pierre Colmez, \emph{Espaces de banach de dimension finie}, Journal of the
  Institute of Mathematics of Jussieu \textbf{1} (2002), no.~3, 331–439.

\bibitem[Far16]{Fargues_geomLL}
Laurent Fargues, \emph{Geometrization of the local langlands correspondence: an
  overview}, arXiv:1602.00999.

\bibitem[FF18]{FF_curve}
Laurent Fargues and Jean-Marc Fontaine, \emph{Courbes et fibr\'es vectoriels en
  th\'eorie de {H}odge p-adique}, Ast\'erisque \textbf{406} (2018).

\bibitem[Hon19]{Hong_quotvb}
Serin Hong, \emph{Classification of quotient bundles on the
  {F}argues-{F}ontaine curve}, arXiv:1904.02918.

\bibitem[Ked08]{Kedlaya_slopefiltrations}
Kiran~S. Kedlaya, \emph{Slope filtrations for relative {F}robenius},
  Ast\'erisque \textbf{319} (2008), 259--301, Repr\'esentations $p$-adiques de
  groupes $p$-adiques. I. Repr\'esentations galoisiennes et
  $(\phi,\Gamma)$-modules.

\bibitem[KL15]{KL15}
Kiran~S. Kedlaya and Ruochuan Liu, \emph{Relative {$p$}-adic {H}odge theory:
  foundations}, Ast\'erisque \textbf{371} (2015), 239.

\bibitem[Sch12]{Scholze_perfectoid}
Peter Scholze, \emph{Perfectoid spaces}, Publications math{\'e}matiques de
  l'IH{\'E}S \textbf{116} (2012), no.~1, 245--313.

\bibitem[Sch18]{Scholze_diamonds}
\bysame, \emph{{\'E}tale cohomology of diamonds}, arXiv:1709.07343.

\bibitem[SW]{SW_berkeley}
Peter Scholze and Jared Weinstein, \emph{Lectures on $p$-adic geometry}.

\end{thebibliography}
	
\end{document}